\documentclass{amsart}
\usepackage{color}

\usepackage{mathrsfs}
\usepackage{amscd,amssymb,amsopn,amsmath,amsthm,graphics,amsfonts,enumerate,verbatim,calc}
\usepackage[all]{xy}
\usepackage[utf8]{inputenc}

\newtheorem{theorem}{Theorem}[section]

\newtheorem{lemma}[theorem]{Lemma}
\newtheorem{proposition}[theorem]{Proposition}
\newtheorem{fact}[theorem]{Fact}

\newtheorem{observation}[theorem]{Observation}
\newtheorem{corollary}[theorem]{Corollary}

\theoremstyle{definition}
\newtheorem{definition}[theorem]{Definition}

\newtheorem{problem}[theorem]{Problem}

\newtheorem{discussion}[theorem]{Discussion}

\newtheorem{hypothesis}[theorem]{Hypothesis}
\newtheorem{acknowledgement}{Acknowledgment}
\theoremstyle{remark}
\newtheorem{remark}[theorem]{Remark}

\newtheorem{notation}[theorem]{Notation}

\newtheorem{construction}[theorem]{Construction}

\DeclareMathOperator{\Tor}{Tor}\DeclareMathOperator{\tor}{tor}
\DeclareMathOperator{\End}{End}

\DeclareMathOperator{\BEnd}{E_b}
\DeclareMathOperator{\Hom}{Hom}

\DeclareMathOperator{\ord}{ord}
\DeclareMathOperator{\AP}{AP}

\DeclareMathOperator{\supp}{supp}
\DeclareMathOperator{\range}{Rang}
\DeclareMathOperator{\NQr}{NQr}

\DeclareMathOperator{\Qr}{Qr}

\newcommand{\Ker}{{\rm Ker}}

\newcommand{\Ext}{{\rm Ext}}

\newcommand{\cl}{{\rm cl}}

\newcommand{\Ord}{{\rm Ord}}

\newcommand{\id}{{\rm id}}

\newcommand{\Rang}{{\rm Rang}}\newcommand{\coker}{{\rm Coker}}

\newcommand{\lh}{{\lg}}
\newcommand{\rest}{{\restriction}}

\newcommand{\rng}{{\rm Rang}}

\newcommand{\sn}{{\smallskip\noindent}}

\newcommand{\cB}{{\mathscr B}}

\newcommand{\cH}{{\mathscr H}}

\newcommand{\bbL}{{\mathbb L}}

\newcommand{\bbP}{{\mathbb P}}

\newcommand{\bbZ}{{\mathbb Z}}

\newcommand{\cf}{{\rm cf}}

\newcount\skewfactor
\def\mathunderaccent#1#2 {\let\theaccent#1\skewfactor#2
\mathpalette\putaccentunder}
\def\putaccentunder#1#2{\oalign{$#1#2$\crcr\hidewidth
\vbox to.2ex{\hbox{$#1\skew\skewfactor\theaccent{}$}\vss}\hidewidth}}

\newcommand{\blueq}[1]{{\color{blue} #1}}

\usepackage{hyperref}

\begin{document}

\title[Co-Hopfian and boundedly endo-rigid groups] {Co-Hopfian and boundedly endo-rigid mixed abelian  groups}

\author[M. Asgharzadeh]{Mohsen Asgharzadeh}

\address{Mohsen Asgharzadeh, Hakimiyeh, Tehran, Iran.}

\email{mohsenasgharzadeh@gmail.com}

\author[M.  Golshani]{Mohammad Golshani}

\address{Mohammad Golshani, School of Mathematics, Institute for Research in Fundamental Sciences (IPM), P.O.\ Box:
	19395--5746, Tehran, Iran.}

\email{golshani.m@gmail.com}

\author[S. Shelah]{Saharon Shelah}

\address{Saharon Shelah, Einstein Institute of Mathematics, The Hebrew University of Jerusalem, Jerusalem,
	91904, Israel, and Department of Mathematics, Rutgers University, New Brunswick, NJ
	08854, USA.}

\email{shelah@math.huji.ac.il}
\thanks{The second author's research has been supported by a grant from IPM (No. 1401030417).  The
	third author would like to thank the Israel Science Foundation (ISF) for partially supporting this research by grant No.
	1838/19, his research partially supported by the grant ``Independent Theories'' NSF-BSF NSF 2051825, (BSF 3013005232). The third author is grateful to an individual who prefers to remain anonymous for providing typing services that were used during the work on the paper. This is publication 1232 of third author.}



\begin{abstract}
For a given cardinal $\lambda$ and a torsion abelian  group $K$ of cardinality less than $\lambda$, we present, under some mild conditions (for example  $\lambda=\lambda^{\aleph_0}$),  boundedly endo-rigid abelian  group   $G$ of cardinality $\lambda$ with $\tor(G)=K$. Essentially, we give a complete characterization of such pairs $(K, \lambda)$. Among other things,  we use a twofold version of the black box. We present an application of the construction of boundedly endo-rigid abelian groups. Namely, we turn to the existence problem of co-Hopfian abelian groups of a given size, and present some new classes of them, mainly in the case of mixed abelian groups. In particular, we give useful criteria to detect when a boundedly endo-rigid abelian group is co-Hopfian and completely determine cardinals $\lambda> 2^{\aleph_{0}}$ for which there is a co-Hopfian abelian group of size $\lambda$.
\end{abstract}

	\date{\today}
\maketitle
\medskip

\tableofcontents


\section {Introduction} \label{0}

By  a torsion (resp. torsion-free)  group we mean an abelian group  such that all its non-zero elements  are of finite  (resp. infinite)  order.  A mixed group $G$ contains
both non-zero elements of finite order and elements of infinite order, and these are connected
via the celebrated short exact sequence
$$(\ast)\qquad\qquad 0\longrightarrow \tor(G)\longrightarrow G\longrightarrow \frac{G}{\tor(G)}\longrightarrow 0.$$
Despite the importances of $(\ast)$,  there are
series of questions concerning   how to glue the issues  from torsion and torsion-free parts and put them together to check the desired  properties for
mixed groups.

Reinhold Baer was interested to find an interplay between abelian  groups and rings, see \cite{B1} and \cite{B2}.
In this regard, he raised the following general problem:
\begin{problem}Which rings can be the endomorphism ring of  a given abelian  group   $G$?\end{problem}
There are a lot of interesting research papers and books that study this problem,
see for example  the books \cite{EM02} and \cite{GT}. According to the recent book of Fuchs \cite{fuchs},
 for mixed groups, only
very little can be said. As an achievement,  we cite the  works
of Corner-G\"obel  \cite{CG85} and Franzen-Goldsmith \cite{fg}.

For any group $G$,  by  $E_f(G)$
we mean the ideal of $\End(G)$ consisting of all elements of $\End(G)$ whose
image is finitely-generated.  In \cite{c}, Corner has constructed an  abelian
group $G:=(M,+)$, for some ring $R$ and  an $R$-module $M$,  such that any of its endomorphisms is of the form multiplication  by some $r\in R$ plus
a distinguished function from $E_f(G)$. One can allow such a distinguished function ranges over other classes such  as finite-range, countable-range, inessential range or even small homomorphism, and there are a lot of work trying to clarify such situations. As a short list, we may mention the papers  Corner-G\"obel  \cite{CG85},      Dugas-G\"obel \cite{dg}, Corner \cite{c}, Thome \cite{t} and  Pierce \cite{p}.

Here, by a bounded group, we mean a group $G$ such that $ nG = 0$ for some fixed  $0< n\in\mathbb{N}$.
By a theorem of Baer and Pr\"ufer a bounded group is a direct sum of
cyclic groups. The converse is not true. However, there is a partial converse
for countable $p$-groups. For more details see the book of Fuchs \cite{fuchs}.
A homomorphism $h\in G_1\to G_2$ of abelian  groups is called
bounded if $\Rang(h)$ is bounded.

\begin{definition} An abelian  group $G$ is  \emph{boundedly rigid} when every endomorphism of it has the form $\mu_n+h$, where $\mu_n$ is multiplication by $n \in \mathbb{Z}$ and $h$ has bounded range.  By $\BEnd(G)$
	we mean the ideal of $\End(G)$ consisting of all elements of $\End(G)$ whose
	image is  bounded.
\end{definition}

Let us explain some motivation.
The concept of a rigid system of torsion-free   groups has a natural analogue for the class of separable  $p$-primary   groups: a family $\{G_i:i\in I\} $ of separable p-primary  groups is called rigid-like if  for all $i\neq j\in I$ every homomorphism  $G_i \to G_j$ is small, and also for all $i\in I$, every endomorphism of $G_i$ is the sum of a small endomorphism and multiplication by a $p$-adic integer.
In his paper \cite{45}, Shelah confirmed a conjecture of Pierce \cite{p} by showing that if $\mu$ is an uncountable strong limit cardinal, then there is
a rigid-like system $\{G_i:i\in I\} $ of separable $p$-primary   groups such that $|G_i | = \mu$
and $|I| = 2^\mu$, see also \cite{sh172} for more results in this direction.

Let us now turn to the paper and state our main results.
 Section 2 contains the preliminaries, basic definitions and notations that we need.
  The reader may skip it, and come back to it when needed
 later.
In Section 3, and as a main result,  we prove the following.

\begin{theorem}\label{1.2}
	Given a cardinal $\lambda$ such that  $\lambda=\lambda^{\aleph_0} > 2^{\aleph_0}$  and a torsion   group $K$ of cardinality less than $\lambda$,
	there is a  boundedly rigid abelian  group $G$ of cardinality $\lambda$
	with $\tor(G)=K.$
\end{theorem}

To prove this, we introduce a series of definitions and present several claims. The first one is the rigidity context, denoted by ${\bold k}$, see Definition \ref{f1}. Also, the main technical tool is a variation of ``\textit{Shelah's black
	box}'', and  we refer to it as   \emph{twofold black box}.  For its definition (resp. its existence), see Definition \ref{m1} (resp. Lemma \ref{m2}). It may be worth to mention that the  black boxes were introduced  by Shelah in   \cite{Sh:227},
where he showed that they follow from ZFC (here, ZFC means the Zermelo--Fraenkel set theory with the axiom of choice).
We can consider black boxes
as  a general method to generate a class of diamond-like principles provable in ZFC.
Then, we continue by introducing the approximation blocks, denoted by $\AP$, more precisely, see  Definition \ref{g2}. There is a distinguished
object  $\bold c$ in $\AP$ that we call it full. The twofold black box, helps us to find  such distinguished objects, see Lemma
\ref{g7}. Here, one may define  the group $G:=G_\bold c$.
Let  $h\in\End(G)$. In order to show $h$ is
boundedly rigid, we apply a couple of reductions (see Lemmas \ref{g12}--\ref{g15}),
to reduce to the case that $h$ factors throughout  $G\to \tor(G)$. Finally, in Lemma \ref{g10} we handle this case,  by showing that any map $G\to \tor(G)$ is indeed boundedly rigid.

In the course of the proof of Theorem \ref{1.2},  we develop  a general  method  that allows us to prove
$0\to \mathbb{Z}\to \End(G)\to \frac{\End(G)}{\BEnd(G)}\to 0$ is exact, and also enables  us to  present a connection to Problem 1.1.
In order to display the connection,
 let $R$ be a ring coming  from the rigidity context. For the propose of the introduction, we may assume that $(R,+)$ is cotorsion-free, see Definition \ref{cot} (with the convenience that the argument becomes easier if we work with $R:=\mathbb{Z}$, or even $(R,+)$ is $\aleph_{1}$-free).
Following our construction, every endomorphism  of $G$ has the form $\mu_r+h$, where $\mu_r$ is multiplication by $r \in R$ and $h$ has bounded range, i.e., the sequence
$$0\longrightarrow  R\longrightarrow \End(G)\longrightarrow \frac{\End(G)}{\BEnd(G)}\longrightarrow 0$$is exact.

\begin{definition}
A group	$G$ is called \emph{Hopfian} (resp. \emph{co-Hopfian}) if its surjective (resp. injective) endomorphisms are automorphisms.
\end{definition}

Essentially, we give complete characterization of the pairs $(K, \lambda)$
by relating our work with the recent   works of Paolini and Shelah, see \cite{1205}, \cite{1214} and \cite{Sh:F2005}.
To this end, first we
recall the following folklore problem:
\begin{problem}\label{p1}
	Construct  co-Hopfian groups of a given size.
\end{problem}

Baer \cite{baer} was the first to investigate Problem \ref{p1} for abelian  groups.
A torsion-free abelian  group is co-Hopfian if and only if it is divisible of finite rank, hence the problem naturally reduces to the torsion and mixed cases. In their important paper  \cite{Be}, Beaumont and Pierce  proved  that if $G$ is co-Hopfian, then $\tor(G)$ is of size at most continuum, and further that $G$ cannot be a $p$-groups of size $\aleph_0$. This naturally left open the problem of the existence of co-Hopfian $p$-groups of uncountable size $\leq 2^{\aleph_0}$, which was later solved by Crawley \cite{crawley} who proved that there exist co-Hopfian $p$-groups of size $2^{\aleph_0}$. Braun and Str\"ungmann \cite{independence_paper} showed that
the existence of three types of infinite abelian  $p$-groups of
size $\aleph_0 < |G| < 2^{\aleph_0}$ are independent of ZFC:
\begin{itemize}
	\item[(a)] both Hopfian and co-Hopfian,
	\item[(b)] Hopfian but not co-Hopfian,
	\item[(c)] co-Hopfian but not Hopfian.
\end{itemize}Also, they proved that the above three
types of groups of size $2^{\aleph_0}$ exist in ZFC.
So, in the light of Theorem \ref{1.2}, the remaining part is $2^{\aleph_{0}}<\lambda<\lambda^{\aleph_{0}}$. Very recently, and among other things,  Paolini and Shelah \cite{1214}
proved that there  is no co-Hopfian
group of size $\lambda$ for such a $\lambda$.
As an application, in Section 4,   we  determine  cardinals $\lambda> 2^{\aleph_{0}}$ for which there is a co-Hopfian group of size $\lambda$. For the precise statement, see Corollary \ref{end}.

Let us recall a connection between  the concepts  boundedly endo-rigid groups and  (co-)Hopfian groups. First, recall from the seminal paper \cite{sh44},  for  any $\lambda $ less than  the first beautiful cardinal,  Shelah proved that  there is an endo-rigid torsion-free group of cardinality $\lambda$.
By definition, for any $f\in\End(G)$ there is $m_f \in \mathbb{Z}$ such that $f(x) = m_fx$. So, $f$ is onto iff $m_f=\pm 1$. In other words, $G$ is Hopfian.
This naturally motives us to detect co-Hopfian property by the help of some boundedly endo-rigid groups. This is what we want to do in \S 4.
Namely,   our first result on co-Hopfian groups is stated as follows:

\begin{construction}
Let $K := \oplus \{\frac{\mathbb{Z}}{p^n\mathbb{Z}} : p \in \mathbb{P} \text{ and } 1 \leq n <m \}$, where  $m< \omega$,
and  $\mathbb{P}$ is the set of prime numbers. Let $G$ be a boundedly endo-rigid abelian group   such that $\tor(G)=K$. Then $G$ is co-Hopfian.	
\end{construction}
W may  recall from Theorem \ref{1.2} that such a group exists for any $\lambda=\lambda^{\aleph_0} > 2^{\aleph_0}$. In fact, the size of $G$ is $\lambda$.

Let $h$ be a natural number. One of the tools that we use is the $h$-power torsion subgroup of $G$:
$$\Gamma_h(G):=\{g\in G:\exists n\in \mathbb {N} \emph{ such that }h^{n}g=0\}.$$The assignment $G\mapsto\Gamma_h(G)$ defines a functor from the category of abelian groups to itself.
It may be worth to mention that, in the style of Grothendieck, this is called  section functor and some authors use $\Tor_h(-)$ to denote it.

In our study of the co-Hopfian property of $G$, the following subset of prime numbers appears: $$ S_{G}: = \{p \in \mathbb{P}:  G / \Gamma_p(G)\emph{ is not  p-divisible}\}.$$
The set $ S_{G}$ helps us to present a useful criteria to detect when a boundedly endo-rigid abelian  group is co-Hopfian:

\begin{proposition}
	Assume $\lambda > 2^{\aleph_{0}}$ and  $G$ is a boundedly endo-rigid abelian  group  of size $\lambda$. Then $G$ is co-Hopfian if and only if:
	
	\begin{enumerate}
		\item[(a):] $S_G$ is a non-empty set of primes,
		
		\item[(b):]
		\begin{enumerate}
			\item[$(b_1)$] $\Gamma_p(G) \neq G,$
			
			\item[$(b_2)$] if $p \in S_G,$ then $ \Gamma_p(G)$ is not bounded,
			
			\item[$(b_3)$] if $\Gamma_p(G)$ is bounded, then it is finite.
		\end{enumerate}
	\end{enumerate}
\end{proposition}

Let  $G$ be an abelian  group. In  order to show that  $G$ is (not) co-Hopfian,  and also to see a connection to bounded morphisms, we introduce  a useful set $\rm{NQr}_{(m, n)}(G)$ consisting of those bounded
$h \in \End(\Gamma_{n}(G))$  such that:
\begin{enumerate}
	\item $h' := m\cdot\id_{\Gamma_{n}(G)} + h \in \End(\Gamma_{n}(G))$ is $1$-to-$1$,
	
	\item $h'$ is not onto or $m > 1$ and $G/\Gamma_{n}(G)$ is not $m$-divisible.
\end{enumerate}	In a series of nontrivial cases  we check $\rm{NQr}_{(m, n)}(G)$ and its negation. This enables us to present   some  new classes of co-Hopfian and non co-Hopfian groups (see  below,  items \ref{a19}--\ref{a31}).

For all unexplained definitions from set theoretic algebra
see
the  books by Eklof-Mekler \cite{EM02}  and G\"{o}bel-Trlifaj
\cite{GT}. Also, for  unexplained definitions from the group theory
see the books of Fuchs \cite{fuchs}, \cite{fuch_vol1} and \cite{fuch_vol2}.

\section{Preliminaries}In this paper all groups are abelian, otherwise specialized.
In this section we recall some basic definitions and facts that will be used for later sections of the paper.
\begin{definition}
	\label{k-free}
	An abelian  group $G$ is called $\aleph_1$-free if every countable subgroup of $G$ is free. More generally, an abelian  group $G$ is called $\lambda$-free if every subgroup of $G$ of cardinality
	$< \lambda$ is free.
\end{definition}

\begin{definition}
	\label{strongly}
	Let $\kappa$   be a regular cardinal. An abelian  group $G$ is said to be
	strongly $\kappa$-free if there is a set $\mathcal{S}$ of $<\kappa$-generated free subgroups of
	$G$ containing {0}  such that for any subset $S$ of $G$ of cardinality
	$<\kappa$   and any $N\in\mathcal{S}$, there is an $L\in\mathcal{S}$  such that $S\cup N\subset L $  and
	$L/N$ is free.
\end{definition}

A  group $G$ is pure in an abelian group $H$ if $G\subseteq H$ and
$nG = nH\cap G$ for every $n\in\mathbb{Z}$. The common notation for this notion
is $G\subseteq_\ast H$.
\begin{fact}\label{pure}
	 Suppose $G$ is a torsion-free group. Then the intersection of pure subgroups of $G$ is again pure.
		In particular, for every $S\subset G$, there
		exists a minimal pure subgroup of $G$ containing $S$. The common notation for this subgroup
		is $\langle S\rangle_G^\ast$.
\end{fact}
\begin{fact}[See {\cite[ Theorem~7]{kap}}]\label{kapla1} Let $G$ be an abelian  group and $H$ a pure and bounded subgroup of $G$.  Then $H$ is a direct summand of $G$.
\end{fact}

The notation $\tor(G) $ stands for the full torsion subgroup
of $G$. There is a natural connection with the functor $\Tor^{\mathbb{Z}}_1(-,\sim)$:$$\tor(G)=\Tor^{\mathbb{Z}}_1(\mathbb{Q}/  \mathbb{Z},G).$$
\begin{fact}[See {\cite[ Theorem~8]{kap}}]\label{kapla2} Let $G $ be an abelian  group and $T \subseteq_* \tor(G)$. If $T$ is the direct sum of a divisible group and a group of bounded exponent, then $T$ is a direct summand of $G$. The same result holds if $T \subseteq_* G.$
\end{fact}

\begin{fact}[See \cite{Be}]\label{co-hop_pgroup}
	\begin{enumerate}
	\item[(i)]  Let $G$ be a countable $p$-group. Then $G$ is co-Hopfian if and only if $G$ is finite.
	
	\item[(ii)] If a group $G$ is co-Hopfian, then $\tor(G)$ is of size at most continuum, and further that $G$ cannot be a $p$-groups of size $\aleph_0$.\end{enumerate}
\end{fact}

\begin{fact}[See {\cite[Theorem~17.2]{fuch_vol1}}]\label{bounded_exponent} If $G $ is a $p$-group of bounded exponent, then $G$ is a direct sum of (finitely many, up to isomorphism) finite cyclic groups.
\end{fact}

\begin{definition}\label{cot} 	\begin{enumerate}
		\item[(i)]  An abelian group $G$ is called
	{\em cotorsion} if $\Ext(J, G) = 0$ for all torsion-free       abelian    groups $J$. 

\item[(ii)]  An abelian group $G$ is called
{\em cotorsion-free} if it has no nonzero co-torsion subgroup.\end{enumerate}\end{definition}

In other words, $G$ is cotorsion provided that it is a direct summand of every
abelian group $H$ containing $G$ with the property that $H/G$ is torsion-free.
Here, we recall  a useful source to produce a cotorsion-free group:
\begin{fact}(See \cite[Corollary 2.10(ii)]{EM02}).
Any $\aleph_1$-free group is	cotorsion-free.
	\end{fact}
The $p$-torsion parts of a group  $G$ are important sources to produce pure subgroups. \begin{notation}Let $ \mathbb{P}$ denote the set of all prime numbers.
	
	\begin{enumerate}
\item[(i)]  Let $p\in \mathbb{P}$. The $p$-power torsion subgroup of $G$
is
$$\Gamma_p(G):=\{g\in G:\exists n\in \mathbb {N} \emph{ such that }p^{n}g=0\}.$$		\item[(ii)]		For each $1 \leq m < \omega$, we let $\Gamma_m(G) := \bigoplus \{\Gamma_p(G) : p \, \vert \, m\}$.
	\end{enumerate}\end{notation}
Recall that the assignment $G\mapsto\Gamma_h(G)$ defines a functor from the category of abelian groups to itself, which is also called  section functor. It has the following important property.  Suppose
$f:G\to H$
 is a homomorphism of abelian groups. Then the following diagram of natural short exact sequences is commutative:

$$
\begin{CD}
0@>>> \Gamma_h(H) @>\subseteq>>H @>>> H/ \Gamma_h(H)  @>>> 0\\
@.@Af\upharpoonright AA @AfAA @A \overline{f}AA   \\
0@>>>\Gamma_h(G) @> \subseteq>>G @>>>  G/ \Gamma_h(G) @>>> 0,\\
\end{CD}
$$where $\overline{f}(g+\Gamma_h(G)):=f(g)+\Gamma_h(H)$.

 The connection from $p$-power torsion functors and the classical torsion functor is read as follows:
$$\Tor^{\mathbb{Z}}_1(\mathbb{Q}/  \mathbb{Z},G)=\tor(G)=\bigoplus_{p\in \mathbb{P}} \Gamma_p(G). $$
\begin{notation}
	In this paper, by  $\End(-)$ we mean $\End_\mathbb{Z}(-)$ where $(-)$ is at least an abelian group, otherwise we specify it.
\end{notation}
The following notion of boundness plays an important role in establishing the main theorems:

\begin{definition}\label{a1}
	Let $G$ be an abelian  group of size $\lambda$. We say $G$ is  \emph{boundedly endo-rigid} when for every $f \in \End(G)$ there is $m \in \bbZ$ such that the map $x \mapsto f(x) - mx$ has bounded range.
\end{definition}
The next fact  follows from the definition.
\begin{fact}\label{a4}
An abelian group 	$G$ is boundedly endo-rigid if and only if for every $f \in \End(G)$ there is $m \in \mathbb{Z}$ and bounded $h \in \End(G)$ such that $f(x) = mx + h(x)$.
\end{fact}

\begin{fact}\label{a7}
	Let $K$ be a bounded torsion abelian group and let $G \subseteq_{*} H$. If $g \in \rm{Hom}(G, K)$, then there is $h \in \rm{Hom}(H, K)$ extending $g$. This property is conveniently summarized by the subjoined  diagram:$$\xymatrix{
		&0 \ar[r]&G\ar[r]^{\subseteq_\ast}\ar[d]_{g}&H\ar[dl]^{\exists h}\\
		&& K
		&&&}$$
\end{fact}

\begin{fact}\label{a10}
	Let $G $ be abelian group and suppose that $G$ is not bounded, then the bounded endomorphisms of $G$ (i.e., those $f \in \End(G)$ with bounded range) form an ideal of the ring $\End(G)$, we denote this ideal by $\BEnd(G)$. With respect to this terminology, $G$ is boundedly rigid if and only if the quotient ring $\End(G)/ \BEnd(G) \cong \mathbb{Z}$.
\end{fact}

\begin{remark}\label{a13} Recall that torsion subgroups are pure. Let $f$ be a bounded endomorphism of  $\tor(G)$.
		By Fact \ref{a7}, we have
	$$\xymatrix{
		&0 \ar[r]&\tor(G)\ar[r]^{\subseteq_\ast}\ar[d]_{f}&G\ar[dl]^{\exists h}\\
		&& \tor(G)
		&&&}$$
Let  $\hat{f}:G\stackrel{h}\longrightarrow \tor(G)\stackrel{\subseteq}\longrightarrow G$. In sum, $f$
	 extends to an endomorphisms $\hat{f}$ of $G$ with the same range:
		$$\xymatrix{
		&& \tor(G)\ar[d]_{\subseteq}\ar[r]^{f}&\tor(G) \ar[d]^{\subseteq}&\\
		&& G\ar[r]_{\hat{f}}& G\\
		&&&}$$
	 Hence, the notion of boundedly rigid is really the right notion of endo-rigidity for mixed groups (for $G$ torsion-free abelian  group, we say that $G$ is endo-rigid when $\End(G) \cong \mathbb{Z}$). For instance, we look at $K = \bigoplus \{\frac{\mathbb{Z}}{p^{\ell+1}\mathbb{Z}}: \ell < m\}$, for some $m<\omega,$ and recall that this has many bounded endomorphisms. The same will happen for any $G$ extending it.
\end{remark}

In what follows we will use the concept of reduced group several times. Let us recall its definition.
\begin{definition}\label{red}Let $G$ be an abelian group.
	\begin{enumerate}
		\item[(a)] $G$ is called reduced if it contains no divisible subgroup  other than $0$.
		
		\item[(b)]  $G$ is called injective if for any inclusion $G_1\subseteq G_2$ of abelian  groups,
	 any morphism $f:G_1\to G$ can be extended into $G_2$:
		$$\xymatrix{
			&0 \ar[r]&G_1\ar[r]^{\subseteq}\ar[d]_{f}&G_2\ar[dl]^{\exists h}\\
			&& G
			&&&}$$	\end{enumerate}
\end{definition}

\begin{fact}\label{ided} (See  \cite{fuchs}). An abelian group $G$ is divisible if and only if it is injective.\end{fact}

Here, we recall a connection between reduced and co-torsion-free abelian groups:

\begin{fact}\label{redco} (See  \cite[theorem V.2.9]{EM02}). An abelian group $G$ is cotorsion-free if and only if it is reduced and torsion-free and does not contain a subgroup
	isomorphic to $\widehat{\mathbb{Z}}_p$ for any prime $p$.\end{fact}Recall that $\widehat{\mathbb{Z}}_p$ means completion of $ \mathbb{Z}$ in the $p$-adic topology.
Here, we collect more basic facts about injective groups that we need:
\begin{discussion}\label{endo}
	Let $p\in \mathbb{P}$ be a prime number. 
	\begin{enumerate}
	\item[(i)] 
	(See \cite[page 11]{EM02}). By the structure theorem for an injective abelian group $I$, we mean the following decomposition:
	$$I=\bigoplus_{p\in \mathbb{P}}\mathbb{Z}(p^\infty)^{\oplus{x_p}}\oplus \mathbb{Q}^{\oplus{x}},$$where ${x_p}$ and $x$ are index sets.
		\item[(ii)] (See \cite[Theorem 3.7]{matlis}).
Let $p,q\in\mathbb{P}_0:=\mathbb{P}\cup\{0\}$, and set $\mathbb{Z}(0^\infty):=\mathbb{Q}$. Then	$$
\Hom(\mathbb{Z}(p^\infty),\mathbb{Z}(q^\infty))=\left\{\begin{array}{ll}
\widehat{\mathbb{Z}}_p&\mbox{if } p=q\\
	0	&\mbox{otherwise }
	\end{array}\right.
	$$with the convenience that $\widehat{\mathbb{Z}}_0=\mathbb{Q}$.

	\item[(iii)] Combining i) and ii) turns out the following well-known formula:$$\End(I)=\prod_{p\in\mathbb{P}_0} \widehat{\mathbb{Z}}_{p}^{\oplus{x_p}},$$where 
	$x_0:=x$.
\end{enumerate}
\end{discussion}

\section{ The ZFC construction of boundedly rigid mixed groups}
\label{s3}
In this section we show that for any  cardinal $\lambda=\lambda^{\aleph_0} > 2^{\aleph_0}$
and any torsion abelian group $K$ of size less than $\lambda,$ there exists a boundedly rigid abelian group $G$
with $\tor(G)=K$, see Theorem \ref{th2}.

To this end, we define the notion of rigidity context $\bold k$ which in particular codes a torsion group $K$, and assign to
it a collection of objects $\bold m$, which among other things have a group $G$ with $\tor(G)=K.$
We show that under the above assumptions on $\lambda$ and $K$, we can always find such an
$\bold m$ such that the associated group  $G$ is boundedly rigid.

\begin{definition}
\label{f1}
\begin{enumerate}
\item We say a  tuple $\bold k$ is a \emph{rigidity context} when
$$\bold k=\big(K_{\bold k}, R_{\bold k}, \phi_r^{\bold k}, \Psi^{\bold k}_{r, s}, \Psi^{\bold k}_{(r, s)},S_{\bold k}\big)_{r,s\in R_{\bold k}}=\big(K, R, \phi_r, \Psi_{r, s}, \Psi_{(r, s)}, S\big)_{r,s\in R}$$
where
\begin{enumerate}
\item[(a)] $K$ is a reduced torsion abelian  group,

 \item[(b)] $R$ is a ring,
\item[(c)]  $S$ is a set of prime numbers, $S^{\bot}_{\bold k}=\mathbb{P} \setminus S$ is its complement, and $R$  is  $S^{\bot}_{\bold k}$-divisible. This means that
 $R$ is divisible for any $p\in S^{\bot}_{\bold k}$,

\item[(d)] for $r \in R$, the map $ \phi_r \in \End(K)$ has bounded range,

\item[(e)] if $r, s \in R,$ then $\Psi_{r, s}=\phi_{r}+\phi_s- \phi_{r+s} \in \End(K)$,

\item[(f)] if $r, s \in R,$ then $\Psi_{(r, s)} \in \End(K)$ has bounded range and, letting $t=rs,$ for $x \in K$ we have
\[
\Psi_{(r, s)}(x)=\phi_r(\phi_s(x)) - \phi_t(x).
\]
\end{enumerate}

\item We say $\bold k$ is nontrivial when for some prime $p\in S_{\bold k}$ the p-torsion $ \Gamma_{p}(K)$ is infinite, or the set $$\{p\in S_{\bold k}:\Gamma_{p}(K)\neq 0\}$$ is infinite.

\item By $\mathbb{Z}_{\bold k}$ we mean the subring of $\mathbb{Q}$ generated by $\{1\} \cup \{\frac{1}{p}:p\in S^{\bot}_{\bold k}\}$.
\end{enumerate}
\end{definition}

\begin{observation}	Suppose $(R_ {\bold k},+)$ is   cotorsion-free as an abelian group. Then
$S_{\bold k}\neq\emptyset$.
\end{observation}\begin{proof} Suppose on the way of contradiction that $S_{\bold k}=\emptyset$. In other words, $S^{\bot}_{\bold k}$  is the set of prime numbers.
By Definition \ref{f1}(1.c), $R$  is  $S^{\bot}_{\bold k}$-divisible. This means that $\mathbb{Q}\subseteq R_ {\bold k}$. It turns out from Fact \ref{redco}   that $(R_ {\bold k},+)$ is not cotorsion-free, a contradiction.\end{proof}
\begin{definition}
\label{f11}
 Let $\bold k$ be  a rigidity context. By   $\bold M_{\bold k}$ we mean the family of all tuples
$$\bold m =\big(\bold k_{\bold m}, G_{\bold m}, F^{\bold m}_r, F^{\bold m}_{r, s}, F^{\bold m}_{(r, s)}\big)_{r, s \in R_{\bold k_{\bold m}}}= \big(\bold k, G, F_r, F_{r, s}, F_{(r, s)}\big)_{r, s \in R_{\bold k}}$$
where
\begin{enumerate}
\item[(a)] $G$ is an abelain group,

\item[(b)] $\tor(G)=K_{\bold k}$,

\item[(c)] for $r \in R_{\bold k},$ $F_r$ is an endomorphism of $G$ extending $\phi^{\bold k}_r$:$$\xymatrix{
	&& K\ar[d]_{\subseteq}\ar[r]^{\phi_r}&K \ar[d]^{\subseteq}&\\
	&& G\ar[r]_{F_r}& G\\
	&&&}$$

\item[(d)] for $r, s \in R_{\bold k},$ $F_{r, s} \in \End(G)$ extends $\Psi {r, s}$
 $$\xymatrix{
	&& K\ar[d]_{\subseteq}\ar[r]^{\Psi {r, s}}&K \ar[d]^{\subseteq}&\\
	&& G\ar[r]_{F_{r, s}}& G\\
	&&&}$$
and they have the same range $F_{r, s}[G]=\Psi_ {r, s}[K]$.

\item[(e)]  for $r, s \in R_{\bold k},$ $F_{(r, s)} \in \End(G)$ extends $\Psi^{\bold k}_{(r, s)}$: $$\xymatrix{
	&& K\ar[d]_{\subseteq}\ar[r]^{\Psi {(r, s)}}&K \ar[d]^{\subseteq}&\\
	&& G\ar[r]_{F_{(r, s)}}& G\\
	&&&}$$
and thereby they have the same range $F_{(r, s)}[G]=\Psi_ {(r, s)}[K]$.

\item[(f)] if $r, s, t \in R$ and $t=r+s$, then for $x \in G$,
\[
F_{r , s}(x)= F_r(x)+F_s(x)-F_{t }(x),
\]

\item[(g)] if $r, s, t \in R$ and $t=rs$, then for $x \in G$,
\[
F_{(r , s)}(x)= F_r(F_s(x))-F_{t }(x).
\]

\end{enumerate}
\end{definition}

\begin{definition}	Adopt the previous notation, and let $$\bold M=\bigcup\big\{\bold M_\bold k:  \bold k \emph{ is a rigidity context} \big\}.$$\begin{enumerate}
	\item[(1)] We define  $\leq _{\bold M}$
	as the following partial order on $\bold M$. Namely,  $\bold m\leq _{\bold M}\bold n $ iff

	\begin{itemize}
		\item[(a)] $\bold m,\bold n\in {\bold M} $,
		\item[(b)] $\bold k_{\bold m}=\bold k_{\bold n}$,
		\item[(c)] $G_{\bold m}\subseteq G_{\bold n}$,
		\item[(d)] $F^{\bold m}_r\subseteq F^{\bold n}_r$.
	\end{itemize}
	\item[(2)] By	$\leq _{\bold M_{\bold k}}$ we mean $\leq _{\bold M}\rest \bold M_{\bold k}$.

\end{enumerate}
\end{definition}

\begin{notation}
	Let $r \in R$ and
	$x \in G_{\bold m}$. By $rx$ we mean $r x:=F^{\bold m}_r(x)\in G_{\bold m} $.
\end{notation}

\begin{definition}
\label{f2}
Suppose $\bold k$ is a rigidity context and $\bold m \in \bold M_{\bold k}$.
\begin{enumerate}
\item We say $\bold m$ is \emph{boundedly rigid} when for every $f \in \End(G_{\bold m})$
there are $r \in R$ and $h \in \End_b(G_{\bold m})$\footnote{so,  $h$ has a bounded range.} and
\[x \in G_{\bold m} \implies f(x)=rx+h(x).
\]
\item We say $\bold m$ is \emph{free} when it has a base $B$ which means that the set
$\{ x+K_{\bold k}: x \in B\}$ is a free base of the abelian  group $G_{\bold m}/K.$

\item We say $\bold m$ is \emph{$\lambda$-free} when $G_{\bold m}/K$ is.

\item We say $\bold m$ is \emph{strongly $\lambda$-free} when $G_{\bold m}/K$ is.

\item  Let $M_{\bold m}$ be the $R$-module obtained by expanding  $G_{\bold m}/K$ such that for $x, y \in G_{\bold m}$ and $r \in R$
\[
rx+K=y+K \iff F^{\bold m}_{r}(x)=y.
\]
\end{enumerate}
\end{definition}
The next easy lemma shows that $M_{\bold m}$ as defined above is well-defined.
\begin{lemma}\label{ms}Suppose $\bold k$ is a rigidity context and $\bold m \in \bold M_{\bold k}$. Then
$M_{\bold m}$ can be turn to an $R$-module structure.
	\end{lemma}
\begin{proof}
	Since $M_{\bold m}$ is an expansion of $G_{\bold m}/K$, it is an abelian group. Let $r\in R$ and $m:=g+K\in M_{\bold m}$ where $g\in G$. The assignment $$(r,m)\mapsto rm:= F^{\bold m}_r(g)+K\in G_{\bold m}/K=M_{\bold m},$$ defines the desired module structure on $ M_{\bold m}$.
\end{proof}

\begin{lemma}
\label{f3}Suppose $\bold k$ is a rigidity context and $\bold m \in \bold M_{\bold k}$. The following assertions hold.
\begin{enumerate}
\item Suppose $R_{\bold k}=\mathbb{Z}$ (so, $S^{\bot}_{\bold k}=\emptyset$). Then $\bold m$
is boundedly rigid iff  $G_{\bold m}$ is boundedly rigid.
\item Let $R_{\bold k}=\mathbb{Z}_{\bold k}$ (see Definition \ref{f1}(3)). Then  $\bold m$
is boundedly rigid iff  $G_{\bold m}$ is boundedly rigid.
\item if $\phi^{\bold k}_r$ is zero for every $r \in R$,
then $G_{\bold m}$ is an R-module.

\end{enumerate}
\end{lemma}
\begin{proof}
(1) and (2) are trivial and follow from the definitions.

(3): For each $x \in G_{\bold m}$ and $r \in R$, we  set $r x:=F^{\bold m}_r(x)$. It is straightforward to furnish the following three properties:
\begin{itemize}
\item the identity $r(x+y)=rx+ry$ follows from Definition \ref{f1}(2)(c),
\item the equality $(r+s)x=rx+sx$ follows from Definition \ref{f1}(2)(d),
\item the equality $r(sm)=(rs)m$ follows from items (e) and (f) from Definition \ref{f1}(2).
\end{itemize}
From these, $G_{\bold m}$ is equipped with an $R$-module structure.
\end{proof}

In what follows, the notation $\lh(-)$ stands for the length function.
\begin{definition}
\label{f5} Let $\alpha \in \Ord.$
\begin{enumerate}
\item  By $\Lambda_\omega[\alpha]$ we mean $$\big\{ \eta:  \lh(\eta)=\omega\emph{ and } \eta(n)=(\eta(n, 1), \eta(n, 2))\emph{ where } \eta(n, 1) \leq \eta(n, 2) <\eta(n+1, 1)< \alpha \big\}.$$

\item For each $\eta\in\Lambda_\omega[\alpha]$, we let $\bold j(\eta)=\bigcup \{\eta(n, 1): n<\omega\}.$
\item $\Lambda_{<\omega}[\alpha]:=\{\langle \rangle\} \cup \bigcup\limits_{k<\omega}\Lambda_k[\alpha]$, where $\Lambda_k[\alpha]$ is the set of all  $\eta$ furnished with the following four properties:
\begin{enumerate}
\item[(a)] $\lh(\eta)=k+1$,

\item[(b)]  $\eta(k)<\alpha$,

\item[(c)] for any $\ell < k $ we suppose  $\eta(\ell)$ is furnished with a pairing property in the following sense:\begin{enumerate}
	\item[$(c.1)$]   $\eta(\ell)=(\eta(\ell, 1), \eta(\ell, 2))$, where $\eta(\ell, 1) \leq \eta(\ell, 2) < \alpha$, and
		\item[$(c.2)$] Suppose in addition $\ell+1<k$, we may and do assume  that $ \eta(\ell, 2) <\eta(\ell+1, 1)$,\end{enumerate}
\item[(d)] if $\ell < k,$ then $\eta(\ell, 1)= \eta(\ell, 2)\Longleftrightarrow\ell=0.$
\end{enumerate}

\item $\Lambda[\alpha]:=\Lambda_\omega[\alpha] \cup \Lambda_{<\omega}[\alpha].$

\item For any $\eta \in \Lambda[\alpha]$ and $k+1 < \lh(\eta)$ we set

\begin{enumerate}
	\item[(5.1)]   $\eta \restriction_L k := \big\langle (\eta(\ell, 1), \eta(\ell, 2)): \ell < k                \big\rangle^{\smallfrown} \langle  \eta(k, 1) \rangle$, and	\item[(5.2)] $\eta \restriction_R k: = \big\langle (\eta(\ell, 1), \eta(\ell, 2)): \ell < k               \big \rangle^{\smallfrown} \langle  \eta(k, 2) \rangle.$\end{enumerate}

Note that $\eta \restriction_L k$ and  $\eta \restriction_R k$ belong to $\Lambda_{k+1}[\alpha].$
\item We say $\Lambda \subseteq \Lambda[\alpha]$ is \emph{downward closed} while for each $\eta \in \Lambda$ and $k+1 < \lh(\eta)$ we have
$\eta \restriction_L k, \eta \restriction_R k \in \Lambda.$
\end{enumerate}
\end{definition}
We next define when a subset of $\Lambda_\omega[\alpha]$
is free.

\begin{definition}
\label{f111} Suppose $\alpha \in \Ord$ and $\Lambda \subseteq  \Lambda_\omega[\alpha]$.
\begin{enumerate}
	\item  	We say $\Lambda$ is free whenever there is a function $h:\lambda\to \omega$ such that the sequence $$\big\langle\{\eta\rest_Ln,\eta\rest_Rn:h(\eta) \leq n<\omega\}: \eta\in\Lambda\big\rangle$$
	is a sequence of pairwise disjoint sets.
	
	\item 	We say $\Lambda$  is $\mu$-free when  every $\Lambda'\subseteq\Lambda$ of cardinality $<\mu$ is free.
\end{enumerate}
\end{definition}

We can now state the main result of this section as follows.
\begin{theorem}
\label{th2} Let $\lambda=\lambda^{\aleph_0} > 2^{\aleph_0}$. Let
$\bold k$ be a nontrivial rigidity context such that $K:=K_{\bold k}$ and $ R:=R_{\bold k}$ are of cardinality  $\leq \lambda$.
Then there exists an abelian  group $G$
such that $\tor(G)=K$
and $G$ is boundedly rigid. In particular, the sequence $$0\longrightarrow   R \longrightarrow \End(G)\longrightarrow \frac{\End(G)}{\BEnd(G)}\longrightarrow 0$$is exact.
\end{theorem}
The rest of this section is devoted to the proof of above theorem.

\begin{definition}
\label{m0}
For any ordinal $\gamma$, a sequence  $\eta \in \Lambda[\lambda]$ and a family $\Lambda \subseteq \Lambda[\lambda]$ we define:
\begin{enumerate}
\item  $S_\gamma$
 is the closure of $\omega \cup \gamma$ under taking finite subsets, so including finite sequences.

\item  $\gamma(\eta)=\eta(0, 1)$.

\item $\Lambda_\gamma=\{\eta \in \Lambda: \gamma(\eta) < \gamma  \}$.


\item We set
\begin{enumerate}
	\item[(4.1)]   $\Lambda_{<\omega}=\Lambda \cap \Lambda_{<\omega}[\alpha]$, and	\item[(4.2)] $\Lambda_{\omega}=\Lambda \cap \Lambda_{\omega}[\alpha].
	$\end{enumerate}
\end{enumerate}
\end{definition}
In order to prove Theorem \ref{th2}, we need a
twofold version of black box, that we now introduce. On simple black boxes see \cite{sh136}, \cite{sh300} and \cite{sh309}.
The presentation here is a special case of the $n$-fold $\lambda$-black box from \cite{511}, when $n=2$.
\begin{definition}
\label{m1} We say $\bold b$ is a \emph{twofold $\lambda$-black box} when it consists of:
\begin{enumerate}
\item $\bar g = \langle   g_\eta: \eta\in  \Lambda_\omega[\lambda]             \rangle$,
where 

\item $g_\eta$ is a function from $\omega$ into $S_{\lambda}$,

\item Suppose $g:\Lambda_{<\omega}[\lambda]\to S_\lambda$ is a function
and $f:\Lambda_{<\omega}[\lambda]\to\gamma$ where $\gamma< \lambda.$ Then for some $\eta \in \Lambda_{\omega}[\lambda]$ the following hold:
\begin{enumerate}
\item
 $\gamma(\eta) > \gamma$,

\item $g_\eta(0)=g(\langle\rangle)$,

\item $g_\eta(n+1)=\big(g(\eta \upharpoonright_L n), g(\eta \upharpoonright_R n)\big)$,

\item  $\eta(n, 1) < \eta(n, 2)$ and $f(\eta \upharpoonright_L n)=f(\eta \upharpoonright_R n)$   for all   $1\leq n < \omega$.
\end{enumerate}
\end{enumerate}
\end{definition}

\begin{hypothesis}
	\label{g1}For the rest of this section we adopt the following hypotheses, otherwise specializes:
	\begin{itemize}
		\item $\lambda=\lambda^{\aleph_0} > 2^{\aleph_0}$.
		\item $\bold k$ is a rigidity context as in Definition \ref{f1}.
		
		\item $K=K_{\bold k}, R=R_{\bold k}$ are of cardinality  $< \lambda$. Without loss of generality, we may assume that the set of elements of $K$ and $R$
		are subsets of $\lambda.$ \item $(R,+)$ is   cotorsion-free.
		
		\item $\bold b$ is a    twofold   $\lambda$-black box.
	\end{itemize}
\end{hypothesis}

The  following result was proved  in \cite[Lemma 1.14]{511}, with a setting more general than here. As this plays a crucial ingredient, we sketch its proof.

\begin{lemma}
\label{m2}
 There exists  a  twofold $\lambda$-black box.
\end{lemma}

\begin{proof} 
For notational simplicity, we set $S:=S_\lambda$, and look at the following fixed partition of $\lambda$ into $\lambda$-many sets, each of cardinality
$\lambda$:$$\big\langle W_{s_1, s_2} : s_1, s_2 \in S \big\rangle.$$ 	For each $\eta \in \Lambda_\omega[\lambda]$, we define $g_\eta(n) \in S$, by induction on $n<\omega.$  

To start, set
\[
(\ast)_1 \qquad\qquad g_\eta(0)=s \iff \eta(0, 1) = \eta(0, 2) \in W_{s, s}.
\]	
Now suppose that $n<\omega$ and $g_\eta \upharpoonright (n+1)$ 
is defined. We are going to define $g_\eta(n+1)$. It is enough to note that 
\[
(\ast)_2 \qquad\qquad g_\eta(n+1)=(s_1, s_2) \iff \eta(n+1, 1) \in W_{s_1, s_2}.
\]
We show that $\bar g = \langle   g_\eta: \eta\in  \Lambda_{\omega}[\lambda]              \rangle$ is as required.
Suppose that  $g:\Lambda_{<\omega}[\lambda]\to S_\lambda$ is a function
and $f:\Lambda_{<\omega}[\lambda]\to\gamma$ where $\gamma< \lambda.$ We define $\eta \in \Lambda_\omega[\lambda]$,
by defining $\eta(n),$ by induction on $n.$ 

Let $\eta(0):=\langle  \eta(0, 1), \eta(0, 2)    \rangle$,
where 
$$(\ast)_3 \qquad\qquad \gamma < \eta(0, 1)= \eta(0, 2) \in W_{g(\langle\rangle), g(\langle\rangle)}.$$
Now, suppose that $n<\omega$ and we have defined $\eta \upharpoonright n+1$. We define $$\eta(n+1)=\langle \eta(n+1, 1), \eta(n+1, 2) \rangle.$$ 
Set 
	\begin{itemize}
	\item[a)]
$s_1:=g(\eta\upharpoonright_L n)$, 
\item[b)] $s_2:= g(\eta \upharpoonright_R n)$,
and\item[c)]  $\bold c_n: W_{s_1, s_2} \rightarrow \gamma$ is defined via the following assignment
\[
\bold c_n(\alpha) := f\big((\eta \upharpoonright n+1) ^{\frown} \langle \alpha \rangle\big)\quad(+)
\]	\end{itemize}
As $\gamma < \lambda$ and $W_{s_1, s_2}$ has size $\lambda,$ we can find an unbounded subset $W_n$ of $W_{s_1, s_2}$ such that $\bold c_n \upharpoonright W_n$ is constant. Let
$ \eta(n+1, 1) < \eta(n+1, 2)$ be such that
\[
(\ast)_4 \qquad\qquad  \eta(n, 2) <  \eta(n+1, 1), \eta(n+1, 2) \in W_n \subseteq W_{g(\eta\upharpoonright_L n), g(\eta\upharpoonright_R n)}.
\]
We claim that the $\eta$ we constructed as above, satisfies the required conditions of Definition \ref{m1}(3). Indeed, thanks to our construction, $\gamma(\eta)=\eta(0, 1) > \gamma.$  We also have
\[
g_\eta(0)=g(\langle\rangle) \iff  \eta(0, 1) = \eta(0, 2) \in W_{g(\langle\rangle), g(\langle\rangle)},
\]
which is true by $(\ast)_3$. We also have
\[
g_\eta(n+1)=\big(g(\eta \upharpoonright_L n), g(\eta \upharpoonright_R n)\big) \iff \eta(n+1, 1) \in W_{g(\eta \upharpoonright_L n), g(\eta \upharpoonright_R n)},
\]
which is again true by $(\ast)_4$. Finally note that, clearly $f(\eta\upharpoonright_L 1)=f(\eta\upharpoonright_R 1)$, and for all $n$,

	\begin{equation*}
\begin{array}{clcr}	f(\eta\upharpoonright_L n+2) &=
f(\eta \upharpoonright n+1^{\frown} \langle \eta(n+1, 1)   \rangle)  \\&\stackrel{(+)}=\bold c_n\big(\eta(n+1, 1)\big)
\\
&\stackrel{(\ast)_4}=\bold c_n\big(\eta(n+1, 2)\big) \\
&\stackrel{(+)}=f(\eta \upharpoonright n+1^{\frown} \langle \eta(n+1, 2)   \rangle)\\
&=f(\eta\upharpoonright_R n+2).
\end{array}
\end{equation*}

The Lemma follows.
\end{proof}

Assuming hypotheses beyond ZFC, we can get stronger versions of twofold $\lambda$-black box (see again \cite{511}).
\begin{observation}
	\label{f6} Assume
	$\lambda=\cf(\lambda) \geq \aleph_1$. Let $$S \subseteq \{\alpha < \lambda: \cf(\alpha)=\aleph_0  \}$$
be a stationary and non-reflecting subset of $\lambda$ such that the principle $\Diamond_S$ holds. Then there is a $\lambda$-free twofold $\lambda$-black box
$\bold b$ such that $\Lambda_{\bold b}=\{\eta_\delta: \delta \in S    \}$ and $\bold j(\eta_\delta)=\delta$ for every $\delta \in S.$
\end{observation}Recall that Jensen's diamond principle $\Diamond_S$ is a kind of prediction principle whose truth is independent of ZFC.
The point in the above proof is that if $\Lambda_{\bold b}=\{\eta_\delta: \delta \in S    \}$ and $\bold j(\eta_\delta)=\delta$ for every $\delta \in S$, then as $S$ does not reflect, the set
$\Lambda_{\bold b}$ is $\lambda$-free.


\begin{remark}
 Recall from \cite{independence_paper} that a (co-)Hopfian  group of size $\lambda=2^{\aleph_0}$ exists in ZFC. We can also deal with the case of $\lambda=2^{\aleph_0}$, but   all is known in this case, so we just concentrate on the case $\lambda=\lambda^{\aleph_0} > 2^{\aleph_0}$.
\end{remark}

\begin{definition}
\label{g2}
 Let $\AP:=\AP_{\bold k, \lambda}$ be the set of all quintuples
$$\bold c=\big(\Lambda_{\bold c}, \bold m_{\bold c}, \Gamma_{\bold c}, X_{\bold c}, \langle a^{\bold c}_{\eta, n}: \eta \in \Lambda_{\bold c}, n<\omega   \rangle\big)$$
such that:
\begin{enumerate}
\item[(a)] $\Lambda_{\bold c} \subseteq \Lambda[\lambda]$ is downward closed.

\item[(b)] $\bold m_{\bold c}  \in \bold M_{\bold k}$. We may write $G_{\bold c}, M_{\bold c}$ instead of $G_{\bold m_{\bold c}}, M_{\bold m_{\bold c}}$ respectively, etc.

\item[(c)] $X_{\bold c}$ is the following set:$$\{rx_{\nu}: r\in R, \nu  \in \Lambda_{\bold c, <\omega}    \}
\cup \{ry_{\eta, n}: r\in R, \eta \in \Lambda_{\bold c, \omega}, n<\omega  \}.$$

\item[(d)]
$G_{\bold c}$ is generated, as an abelian  group, by the sets $K$ and
 $X_{\bold c}.$ The relations presented in item (f), see below.

\item[(e)] for any ordinal $\alpha$, let $G_{\bold c, \alpha}$ be the subgroup of $G_{\bold c}$
generated by the set $K$ and
$$\{rx_{\nu}: r\in R, \nu  \in \Lambda_{\bold c, <\omega} \cap \Lambda[\alpha]   \}
\cup \{ry_{\rho, n}: r\in R, \rho \in \Lambda_{\bold c, \omega}  \cap \Lambda[\alpha], n<\omega  \}.$$

\item[(f)] $M_{\bold c}$, as an $R$-module,  is generated  by $X_{\bold c}\cup K$, freely except the following set $\Gamma_{\bold c}$
of equations:\begin{itemize}
	\item
 $y_{\eta, n}=a^{\bold c}_{\eta, n}+(n!)y_{\eta, n+1}+(x_{\eta\rest_Ln}-x_{\eta\rest_Rn}),$ \end{itemize}
where $a^{\bold c}_{\eta, n}\in G_{\bold c, \eta(0, 1)}.$
\end{enumerate}
\end{definition}
 The following is clear:
\begin{lemma}\label{siz}
	Suppose $\bold c \in \AP_{\bold k,  \lambda}$. Then $G_{\bold c}$ is of size $\lambda^{\aleph_0}$.
	\end{lemma}

\begin{definition}\label{circ}
For any $\bold c \in \AP_{\bold k, \lambda}$, we define the following:
\begin{enumerate}
\item $\gamma_{\bold c}:=\min\{ \gamma \leq \lambda: \Lambda_{\bold c} \subseteq \Lambda[\gamma] \}$.

		\item  Let $\Omega_{\bold c}:=\Lambda_{\bold c,<\omega}\cup
\big(\Lambda_{\bold c,\omega}\times \omega\big)$
and define $\langle
	x_{\rho}:\rho\in \Omega_{\bold c}	\rangle$ by the following rule	\begin{enumerate}
		\item[(2.1)]  If $\rho \in\Lambda_{\bold c,<\omega}$, then $x_\rho$ is defined as in Definition \ref{g2}(c).
		\item  [(2.2)] If $\rho=(\eta, n) \in \Lambda_{\bold c,\omega}\times \omega $ ,
		we define $x_\rho:=y_{\eta,n}$.	\end{enumerate}	\item
	For $b\in G_{\bold c}$ choose the sequence $$\big\langle r_{b,\ell},\eta_{b,\ell},m_{b,\ell}:\ell<n_b\big\rangle$$  such that
	$$b-\sum_{\ell<n_b} r_{b,\ell}y_{\eta_{{b,\ell}},m_{b,\ell}}\in\sum_{\rho\in\Lambda_{\bold c, < \omega}}Rx_\rho +K, $$where $r_{b,\ell}\in R \setminus \{0\}$ and $(\eta_{b,\ell}, m_{b,\ell})\in \Lambda_{\bold c,\omega} \times \omega$.
\item By $\supp_\circ(b)$ we mean $\{\eta_{{b,\ell}}:\ell<n_b\}$.

	\end{enumerate}	
\end{definition}

\begin{definition}
	\label{g211}
	Suppose $\bold c \in \AP_{\bold k,  \lambda}$ and let $a \in G_{\bold c}$.
	\begin{enumerate}

\item[(a)]
There is a finite set $\Lambda_{a}\subseteq \Lambda_{\bold c}$, a sequence $S:=\langle r_\rho:\rho\in\Lambda_{a}\rangle$ of non-zero elements of $R$,  an $n(a)<\omega$ and $d_a\in K$ such that \begin{equation*}
\begin{array}{clcr}
a  =   \sum\limits_{ \eta\in \Lambda_{a,<\omega}}r_\eta x_\eta
+  \sum\limits_{\nu\in \Lambda_{a,\omega}}r_\nu y_{\nu,n(a)} +d_a,
\end{array}
\end{equation*}
where $\Lambda_{a,<\omega}=\Lambda_{a}\cap \Lambda_{\bold c,<\omega} $ and $\Lambda_{a,\omega}=\Lambda_{a}\cap \Lambda_{\bold c,\omega} $.

\item[(b)]
Let
$\supp_{\bold c}(a)=\supp(a)$ be the minimal set $\Lambda \subseteq \Lambda_{\bold c}$ minimal with respect to the following two properties:
\begin{enumerate}
	
	\item[(b.1)]$\Lambda_a \subseteq \Lambda$.	\item[(b.2)] If $\nu\in\Lambda_{a}\cap \Lambda_{\bold c,\omega}$ and $n<\omega$ then
	$\Lambda_{a^{\bold c}_{\nu,n}}\subset \Lambda$ and $\eta\rest_Ln,\eta\rest_Rn\in\Lambda$.\end{enumerate}	

\end{enumerate}
\end{definition}

\begin{remark}
Adopt the previous notation,	
and $a \in G_{\bold c}$. Then
$\supp_{\bold c}(a)$ is the minimal set $\Lambda \subseteq \Lambda_{\bold c}$ such that
\[
a \in \big\langle \{x_\eta, y_{\nu, n}: \eta \in \Lambda(L, R), \nu \in \Lambda, n<\omega    \} \cup K      \big\rangle^*_{G_{\bold c}}.
\]	
\end{remark}

\begin{remark}Adopt the previous notation. The following holds.
\begin{enumerate}
		\item[(1)] The set $\supp_{\bold c}(a)$ is countable.
	\item[(2)] If $a=x_\nu$ for some $\nu\in\Lambda_{\bold c}$, then $$\supp(a)\setminus S_{\eta(\nu,1)}=\{\nu\}\cup\{\nu\rest_L,n,\nu\rest_R,n:n<\omega\}.$$
\end{enumerate}
\end{remark}

\begin{definition}
\label{g2110}
 Let $\leq_{\AP}$ be the following partial order on $\AP=\AP_{\bold k, \lambda}.$ For any $\bold c,\bold d\in {\AP} $ we say $\bold c \leq_{\AP} \bold d$ when the following holds:
\begin{enumerate}
\item[(a)] $\Lambda_{\bold c} \subseteq \Lambda_{\bold d}$,

\item[(b)] $\bold m_{\bold c}\leq _{\bold M}\bold m_{\bold d} $, hence $G_{\bold c} \subseteq G_{\bold d},$ etc.

\item[(c)] $a^{\bold c}_{\eta, \ell}=a^{\bold d}_{\eta, \ell}$ for $\eta \in \Lambda_{\bold c}, \ell < \omega$,

\item[(d)] $x^{\bold c}_\eta=x^{\bold d}_\eta$ for $\eta \in \Lambda_{\bold c, < \omega}$,
 \item[(e)] $y^{\bold c}_{\eta, \ell}=y^{\bold d}_{\eta, \ell}$
for $\eta \in \Lambda_{\bold c, \omega}$ and $\ell < \omega$.
\end{enumerate}
\end{definition}

\begin{lemma}
\label{g4} The following two assertions are valid:
\begin{enumerate}
\item $\leq_{\AP}$ is indeed a partial order,

\item If $\bar{\bold c}=\langle  \bold c_\alpha: \alpha < \delta    \rangle$ is $\leq_{\AP}$-increasing, then there exists  $\bold c_\delta=\bigcup\limits_{\alpha < \delta} \bold c_\alpha$
   in $\AP$ which is the $\leq_{\AP}$-least upper bound of the sequence $\bar{\bold c}$.
\end{enumerate}
\end{lemma}
\begin{proof}
Clause (1) is clear, for clause (2), let $$\bold c_\delta:=\big(\Lambda, \bold m, \Gamma, X, \langle a_{\eta, n}: \eta \in \Lambda, n<\omega \rangle\big),$$
where:
\begin{itemize}
\item $\Lambda:=\bigcup\limits_{\alpha < \delta}\Lambda_{\bold c_\alpha}$,

\item $\bold m=:(G, F_r, F_{r, s}, F_{(r, s)})$, where
\begin{itemize}
\item $G:=\bigcup\limits_{\alpha < \delta}G_{\bold c_\alpha}$,
\item $F_r := \bigcup\limits_{\alpha < \delta}F_r^{\bold c_\alpha}$, 
\item $F_{r, s} := \bigcup\limits_{\alpha < \delta}F_{r, s}^{\bold c_\alpha}$,
 \item $F_{(r, s)} := \bigcup\limits_{\alpha < \delta}F_{(r, s)}^{\bold c_\alpha}$.
\end{itemize}
\item $\Gamma := \bigcup\limits_{\alpha < \delta}\Gamma_{\bold c_\alpha}$,

\item $X:=\bigcup\limits_{\alpha < \delta}X_{\bold c_\alpha}$,

\item for $\eta \in \Lambda_\omega$ and $n<\omega$, we have $a_{\eta, n}= a^{\bold c_\alpha}_{\eta, n}$, for some and hence any
$\alpha < \delta$ such that $\eta \in \Lambda_{\bold c_\alpha, \omega}$.
\end{itemize}
It is easily seen that $\bold c_\delta$ is as required.
\end{proof}

An $R$-module $M$ is called $\aleph_{1}$-free, if every countably generated submodule of  $M$  is contained in a free submodule of  $M$. Similarly, $\mu$-free can be defined.
For more details, see \cite[IV. Definition 1.1]{EM02}.
\begin{lemma}
\label{g3} Let $\bold c \in \AP$.  The following claims hold:
\begin{enumerate}
\item $\tor(G_{\bold c})=K$. \item
The group $$G_{\bold c}/\big\langle K\cup \{rx_{\nu}:r\in R, \nu\in\Lambda_{\bold c,<\omega}\}\big\rangle$$
is divisible and torsion-free.
Also, the parallel result holds for the $R$-module: $$M_{\bold c}/\big\langle K\cup \{rx_{\nu}:r\in R, \nu\in\Lambda_{\bold c,<\omega}\}\big\rangle.$$

\item The following three properties are satisfied:
\begin{enumerate}
\item $\Lambda_{\bold c} $ is $\aleph_{1}$-free.

\item If $\Lambda_{\bold c}$ is $\mu$-free, then $M_{\bold c}$ is $\mu$-free.

\item  If $\Lambda_{\bold c}$ is $\mu$-free
and $(R,+)$ is $\mu$-free, then $G_{\bold c}/K$ is a $\mu$-free abelian group.
\end{enumerate}
\item If   $\gamma \leq \gamma_{\bold c}$
and $\Lambda \subseteq \Lambda_{\bold c}$, then there exists a unique $\bold d \in \AP$
 such that

\begin{enumerate}
	
	\item[(a)] $\Lambda_{\bold d}=\Lambda \cap \Lambda[\gamma]$,
	
	\item[(b)] $G_{\bold d} \subseteq G_{\bold c}$.
\end{enumerate}
 Such a unique object is denoted  by $\bold d:=\bold c \restriction (\gamma, \Lambda).$

\item Assume $\eta \in \Lambda_\omega[\lambda]\setminus \Lambda_{\bold c}$,
  $\ell < \omega$ and $ a_\ell \in G_{\bold c}$ are such that
$a_\ell \in G_{\bold c, \eta(0, 1)} $ for each $\ell$.
Then there is $\bold d \in \AP$ equipped with the following three properties:
\begin{enumerate}
\item[(a)] $\Lambda_{\bold d}=\Lambda_{\bold c} \cup \{\eta\}\cup\{\eta\rest_Ln,\eta\rest_Rn: n<\omega\},$

\item[(b)] $\bold c \leq_{\AP} \bold d$ and so $G_{\bold c} \subseteq G_{\bold d}$,

\item[(c)] $a^{\bold d}_{\eta, \ell}=a_\ell$ for $\ell < \omega.$
\end{enumerate}
\item The group $G_{\bold c}$ is of size $\lambda$.
\end{enumerate}
\end{lemma}
\begin{proof}
(1)-(2):
 These are easy.

  (3): $(a):$ Let $\Lambda\subseteq\Lambda_{\bold c, \omega}$ be countable, and let $\{\eta_n:n<\omega\}$
be an enumeration of it.
 Define the maps $h_1$ and $h_2$ from $\Lambda$ to $\omega$ as follows:
 \[
 h_1(\eta_n):= \min\bigg\{k:\forall j<  n, \emph{ } \forall \ell,\emph{ }r\in\{L,R\}  \emph { we have }{\eta_j \upharpoonright_\ell k} \neq{\eta_n \upharpoonright_r k}\bigg\},
 \]
 and
 \[
 h_2(\eta_n):= \min\bigg\{k: \eta_n {\upharpoonright_L} k \neq{\eta_n \upharpoonright_R k}\bigg\}.
 \]
 Finally, we set
 \[
 h(\eta_m):= \max\{h_1(\eta_n), h_2(\eta_n)\}+1.
 \]
Having Definition \ref{f111} in mind, we are going to show $h$ is as required.
Let $j < i<\omega$ and let
\begin{itemize}
	\item   $h(\eta_{j})\leq n_j<\omega$
	
	\item  $h(\eta_{i})\leq n_i<\omega$.

\end{itemize}
We will show that $\eta_j \upharpoonright_\ell {n_i}  \neq{\eta_i \upharpoonright_r {n_j}}$,  where $\ell,r\in\{L,R\}$. To see this, we note that there is nothing to prove if ${n_i}\neq{n_j}$. So, we may and do assume that $n:={n_i}={n_j}$. Thus, $h(\eta_{j}),h(\eta_{i})\leq n$. W look at $m:= h_1(\eta_{i})$. According to the definition of $h_1,$
 we know that ${\eta_j \upharpoonright_\ell m} \neq{\eta_i \upharpoonright_r m}$. As  $m\leq n$ one has $${\eta_i \upharpoonright_\ell n} \neq{\eta_j \upharpoonright_r n}.$$
 Also given any $i<\omega,$ if $n \geq h(\eta_i)$, then by the definition of $h_2$ and as $n \geq h_2(\eta_i),$ we have
 $$ \eta_i {\upharpoonright_L} n \neq{\eta_i \upharpoonright_R n}.$$
It follows that the sequence $$\big\langle\{\eta\rest_Ln,\eta\rest_Rn:h(\eta) \leq n<\omega\}: \eta\in\Lambda\big\rangle$$
	is a sequence of pairwise disjoint sets.
By definition, $\Lambda_{\bold c} $ is $\aleph_{1}$-free.

$(b):$ For simplicity, we present the proof when $\mu:=\aleph_{1}$. Let $X\subseteq M_{\bold c}$ be countable. We are going to show that it is included into a countably generated free $R$-submodule of $M_{\bold c}$.
As $X$ countable,
\begin{itemize}
	\item   $\exists \Lambda \subseteq \Lambda_{\bold c,   \omega}$ countable,
	
	\item   $\exists \Lambda_\ast \subseteq \Lambda_{\bold c, <  \omega}$ countable
	
\end{itemize}
such that $$X\subseteq \sum\{Ry_{\eta,n}:\eta\in\Lambda \emph{ and }n<\omega \}+ \sum\{Rx_\rho:\rho\in\Lambda_\ast\}.$$
As $\Lambda_{\bold c}$ is $\aleph_1$-free
and $\Lambda$ is countable,  there is a function $h:\Lambda\to \omega$ such that the sequence $$\big\langle\{\eta\rest_Ln,\eta\rest_Rn:h(\eta) \leq n<\omega\}: \eta\in\Lambda\big\rangle$$
 is a sequence of pairwise disjoint sets.
 Now, we note the following two properties:
\begin{enumerate}
	\item[$(b)_1$:] The $R$-module $$M_\Lambda:=\big\langle x_{\eta\rest_Ln},x_{\eta\rest_Rn}, y_{\eta,n}:\eta\in\Lambda:h(\eta) \leq n<\omega\big\rangle $$
	is free;
	\item[$(b)_2$:] Set $M_{\Lambda\cup\Lambda_\ast}:=\big\langle M_{\Lambda} \cup \{ x_\nu:\nu\in\Lambda_\ast\}\big\rangle $. Then the $R$-module  $M_{\Lambda\cup\Lambda_\ast}/M_{\Lambda_\ast} $
	is free.
\end{enumerate}
In view of $(b)_2$ the short exact sequence  $$ 0\longrightarrow {M_{\Lambda}} \longrightarrow {M_{\Lambda\cup\Lambda_\ast}}\longrightarrow  M_{\Lambda\cup\Lambda_\ast}/M_{\Lambda} \longrightarrow 0,$$splits. Combining this along with  $(b)_1$,
	we observe that $ M_{\Lambda\cup\Lambda_\ast}	$ is free. Since it includes $X$, we get the desired claim.

$(c):$
Now, suppose $(R,+)$ is $\mu$-free.
Let $H$ be a subset of  $(G_{\bold c}/K,+)$ of size $<$ $\mu$. There is a free $R$-module $F$ such that $H\subset F$. There is a subset $S$ of $R$ of size $<$ $\mu$ such that any element of $H$ can be written from a linear combination from $F$ with coefficients taken from $S$. As $(R,+)$ is $\mu$-free, there is a free subgroup $(T,+)$
of it containing $S$. In other words, $$H\subseteq T\ast F:=\bigg\langle\sum_{}\{t_if_i:t_i\in T,f_i\in F\}\bigg\rangle.$$ Since  $(T\ast F,+)$  is free as an abelian group, we get the desired claim.

(4): Let $\bold d$ be such that:
\begin{itemize}
\item[4.1)] $\Lambda_{\bold d}=\Lambda \cap \Lambda[\gamma],$

\item[4.2)]  $X_{\bold d}$ is defined using $\Lambda_{\bold d}$ naturally,

\item[4.3)] for $\nu \in \Lambda_{\bold d, \omega}$ and $n<\omega$,  $a^{\bold d}_{\nu, n}=a^{\bold c}_{\nu, n}$,

\item[4.4)] $\Gamma_{\bold d}$ is defined naturally as the set of equations in (1), but only for
$\eta \in \Lambda_{\bold d, \omega}$.
\end{itemize}
This is straightforward to check that  $\bold d$ is as required.

(5): Let $\bold d$ be defined in the natural way, so that:
\begin{itemize}
\item[5.1)]  $\Lambda_{\bold d}=\Lambda_{\bold c} \cup \{\eta\}\cup\{\eta\rest_Ln,\eta\rest_Rn: n<\omega\},$

\item [5.2)] $X_{\bold d}=X_{\bold c} \cup \{ x_{\eta \upharpoonright_L n}, x_{\eta \upharpoonright_R n}: n<\omega    \}
\cup \{ y_{\eta, n}: n<\omega   \}$,

\item[5.3)] for $\nu \in \Lambda_{\bold c, \omega}$ and $n<\omega$,  $a^{\bold d}_{\nu, n}=a^{\bold c}_{\nu, n}$,

\item [5.4)]$a^{\bold d}_{\eta, n}=a_n$ for $n < \omega,$

\item[5.5)] in addition to the equations displayed in $\Gamma_{\bold c}$, $\Gamma_{\bold d}$ contains   equations of the following forms
\[
y_{\eta, n}=a_n+(n!)y_{\eta, n+1}+(x_{\eta\rest_Ln}-x_{\eta\rest_Rn}),
\]
where $n<\omega.$
\end{itemize}
The assertion is now obvious by  the above definition of $\bold d$.

(6). In view of Lemma \ref{siz}, the group $G_{\bold c}$ is of size $\lambda^{\aleph_0}$. Recall from Hypothesis
\ref{g1} that $\lambda^{\aleph_0} =\lambda$. So, the desired claim is clear.
\end{proof}

\begin{lemma}
	\label{g31} Let $\bold c \in \AP$. Then
	the abelian group $G_{\bold c}/K$ is reduced.
\end{lemma}
\begin{proof}
Suppose on the way of contradiction that $G_{\bold c}/K$ is not reduced. Then
	it has  a divisible direct summand, say $I$. By Fact \ref{ided}, $I$ is injective. We apply the structure theorem for injective abelian groups  (see Discussion \ref{endo}(i)) to find the following  decomposition:
	$$I=\bigoplus_{p\in \mathbb{P}}\mathbb{Z}(p^\infty)^{\oplus{x_p}}\oplus \mathbb{Q}^{\oplus{x}},$$where ${x_p}$ and $x$ are index sets.
Since  	$G_{\bold c}/K$ is torsion-free, $I$ is
torsion-free. So, $I$ has no $p$-torsion part.  This shows that
${x_p}=\emptyset$ for all $p\in \mathbb{P}$.
In other words, $I= \mathbb{Q}^{\oplus{x}}.$ Since $I$ is nonzero, ${x}\neq\emptyset$.
 This yields that $(\mathbb{Q},+)$ is a directed  summand of $G_{\bold c}/K$.
Thanks to Lemma \ref{g3}(3)(a)
$\Lambda_{\bold c}$ is $\aleph_1$-free.
We combine this with Lemma \ref{g3}(3)(b)
to deduce that $M_{\bold c}$ is $\aleph_1$-free as an $R$-module.

	 We have two possibilities: 1) $\bold k$ is trivial, and  2)
	$\bold k$ is nontrivial.
	
	1) $\bold k$ is trivial: Then $R:=\mathbb{Z}$. Recall
	that $M_\bold c=G_{\bold c}/K$ is $\aleph_{1}$-free. Since   $(\mathbb{Q},+)$ is countable, it should be free, a contradiction.

	2) $\bold k$ is nontrivial:
	Recall that $R$  is  $S^{\bot}_{\bold k}$-divisible. Since the context is nontrivial,  there is $p\in S^{\bot}_{\bold k}$ such that $\{1/p^n:n\gg 0\}\subseteq R$. For simplicity,
	we assume that $\{1/p^n:n> 0\}\subseteq R$.
	Since  $M_{\bold c}$ is $\aleph_1$-free and that $\{1/p^n:n> 0\}\subseteq \mathbb{Q}\subseteq M_\bold c$, there is a free $R$-module $F\subseteq  M_\bold c$ such that  $\{1/p^n:n> 0\}\subseteq F$. Let $F=\bigoplus R$.
So, the desired contraction follows by:
	\begin{equation*}
	\begin{array}{clcr}	\{r/p^n:n> 0,r\in R\}&=
\bigcap_{\ell>0}p^\ell	\{r/p^n:n> 0,r\in R\} \\&\subseteq \bigcap_{\ell>0} p^\ell F
	\\
	&=  \bigoplus(\bigcap_{\ell>0} p^\ell R) \\
	&\subseteq \bigoplus(\bigcap_{\ell>0}  \ell R)\\
	&=0,
	\end{array}
	\end{equation*}
where   the last equality comes from the fact that $(R,+)$ is cotorsion-free, in fact by Fact \ref{redco}, the abelain group $(R,+)$ is reduced, and so $\bigcap_{\ell>0}  \ell R =0$.
The proof is now complete.
\end{proof}

The following easy lemma will be used later at several places.

\begin{lemma}
\label{g21} Let $\bold c \in \AP_{\bold k, \lambda}$. Then  the following equation	\begin{equation*}
\begin{array}{clcr}
y^{\bold c}_{\eta, 0}&= \sum\limits_{i=0} ^n \big(\prod_{j<i}j!\big) a^{\bold c}_{\eta, i}+(\prod_{i=1}^{n}i!) y^{\bold c}_{\eta, n+1}+ \sum\limits_{i=0}^n \big(\prod_{j<i}j!\big)(x^{\bold c}_{\eta \restriction_L i}-x^{\bold c}_{\eta \restriction_R i}),
\end{array}
\end{equation*}is valid for any $n<\omega$.
\end{lemma}
\begin{proof}
We proceed by induction on  $n$. The desired claim is clearly holds for $n=0.$ Suppose inductively that it holds for $n$. We are going to show the claim for $n+1$. To this end, we apply the induction assumption along with the relation
	$$y^{\bold c}_{\eta, n+1}=a^{\bold c}_{\eta, n+1}+(n+1)!y^{\bold c}_{\eta, n+2}+ (x^{\bold c}_{\eta \restriction_L n+1}-x^{\bold c}_{\eta \restriction_R n+1})$$
 to deduce\begin{equation*}
	\begin{array}{clcr}
	y^{\bold c}_{\eta, 0}&= \sum\limits_{i=0} ^n \big(\prod_{j<i}j!\big) a^{\bold c}_{\eta, i}+\big(\prod_{i=1}^{n}i!\big)y^{\bold c}_{\eta, n+1}+ \sum\limits_{i=0}^{n+1} (x^{\bold c}_{\eta \restriction_L i}-x^{\bold c}_{\eta \restriction_R i})
	\\
	&= \sum\limits_{i=0} ^n \big(\prod_{j<i}j!\big) a^{\bold c}_{\eta, i}+\big(\prod_{i=0}^n i!\big)a^{\bold c}_{\eta, n+1}+\big(\prod_{i=1}^{n}i!\big)(n+1)!y^{\bold c}_{\eta, n+2}\\
&~\quad~+ \big(\prod_{i=0}^n i!\big)(x^{\bold c}_{\eta \restriction_L n+1}-x^{\bold c}_{\eta \restriction_R n+1})+\sum\limits_{i=0}^n \big(\prod_{j<i}j!\big)(x^{\bold c}_{\eta \restriction_L i}-x^{\bold c}_{\eta \restriction_R i})\\
	&=\sum\limits_{i=0}^{n+1} \big(\prod_{j<i}j!\big) a^{\bold c}_{\eta,i}+\big(\prod_{i=1}^{n+1}i!)y^{\bold c}_{\eta, n+2}+ \sum\limits_{i=0}^{n+1} \big(\prod_{j<i}j!\big)(x^{\bold c}_{\eta \restriction_L i}-x^{\bold c}_{\eta \restriction_R i}).
	\end{array}
	\end{equation*}
	Thus the claim holds for $n+1$ as well.
\end{proof}
There are some distinguished and useful
objects in $\AP_{\bold k,  \lambda}$:

\begin{definition}
\label{g6} We say $\bold c \in \AP_{\bold k, \lambda}$ is full when:
\begin{enumerate}
\item[(a)] $\Lambda_{\bold c} \supseteq \Lambda_{<\omega}[\lambda]$,

\item[(b)] if $a_n \in G_{\bold c}$ for $n<\omega$ and $f: \Lambda_{<\omega}[\lambda] \to \gamma$, where $\gamma < \lambda,$ then for some $\eta \in \Lambda_{\bold c}$
and all $n<\omega$ we have
$a^{\bold c}_{\eta, n}=a_n$ and $f(\eta \restriction_L n)=f(\eta \restriction_R n)$.

\end{enumerate}
\end{definition}Now, we study the existence problem for  fullness    in $\AP$:
\begin{lemma}
\label{g7}Adopt the notation from Hypothesis
\ref{g1}.
Then
there are some full $\bold c \in \AP_{\bold k, \lambda}.$
\end{lemma}
\begin{proof}
	Let $\bold b$ be a twofold $\lambda$-black box, which exists by Lemma \ref{m2}.
We look at $$\Omega:=\Lambda_{<\omega}[\lambda]\cup
	\big(\Lambda_{\omega}[\lambda]\times \omega\big),$$  and for each ordinal $\alpha < \lambda$ we set $$\Omega_\alpha:=\Lambda_{<\omega}[\alpha]\cup
	\big(\Lambda_{\omega}[\alpha]\times \omega\big).$$
 Fix a bijection map $$h:S_{\lambda}\stackrel{\cong}\longrightarrow (\oplus_{\rho\in \Omega} Rx_\rho)\oplus K $$such that for each ordinal $\alpha< \lambda$ one has $$h''[S_\alpha]\subseteq  (\oplus_{\rho\in \Omega_\alpha} Rx_\rho)\oplus K \quad(\ast),$$
 This is possible, as for each $\alpha,$
 \[
 |S_\alpha| \leq \aleph_0+|\alpha| \leq  |(\oplus_{\rho\in \Omega_\alpha} Rx_\rho)\oplus K| < \lambda.
 \]
  Let $\bold c$ be defined as
 \begin{enumerate}
 	\item $\Lambda_{\bold c}=\Lambda_{\omega}[\lambda] \cup \Lambda_{<\omega}[\lambda].$
 	
 	\item $X_{\bold c}$ is the following set:$$\{rx_{\nu}: r\in R, \nu  \in \Lambda_{\bold c, <\omega}    \}
 	\cup \{ry_{\eta, n}: r\in R, \eta \in \Lambda_{\bold c, \omega}, n<\omega  \}.$$
 	
 \item $a^{\bold c}_{\eta, n}=h(g_{\eta}^\bold b(n+1))$, where $g_{\eta}^\bold b$ is given by the twofold $\lambda$-black box.
 	\item
 	$G_{\bold c}$ is generated, as an abelian  group, freely by the sets $K$ and
 	$X_{\bold c}$ except the following set of relations:
 	
 	\begin{equation*}
 	\begin{array}{clcr}
 	y_{\eta, n}&= a^{\bold c}_{\eta, n}+(n!)y_{\eta, n+1}+(x_{\eta\rest_Ln}-x_{\eta\rest_Rn}),
 	\end{array}
 	\end{equation*}
 with the convenience that $a^{\bold c}_{\eta, n}$ is regarded as an element of $G_\bold c$ via the quotient map $$(\bigoplus_{\rho\in \Omega} Rx_\rho)\oplus K \twoheadrightarrow G_\bold c.$$
From this identification and  (*), $a^{\bold c}_{\eta, n} \in G_{\bold c, \eta(0, 1)}.$

 	\item $\Gamma_{\bold c}$ is defined naturally as in Definition \ref{g2}.
 \end{enumerate}

Let us show that $\bold c$ is as required. It clearly satisfies clause (a) of Definition \ref{g6}. To show that
clause (b)
of Definition \ref{g6} is satisfied,
let  $\langle a_n:n<\omega\rangle\in~   ^{\omega} G_{\bold c}$ and  $f: \Lambda_{<\omega}[\lambda] \to \gamma,$ where $\gamma < \lambda$.
Let  $g: \Lambda_{<\omega}[\lambda]\to S_\lambda$ be defined such that
 for all $\nu \in  \Lambda_{<\omega}[\lambda]\setminus \{ \langle\rangle \}$, $$h(g(\nu))=a_{\lg(\nu)-1}\quad(+).$$
We are going to apply the twofold $\lambda$-black box $\bold b$. According to its properties, there is an $\eta \in \Lambda_{\omega}[\lambda]$
such that:
\begin{enumerate}
\item[(6)]
 $\gamma(\eta) > \gamma$,

\item[(7)] $g^{\bold b}_\eta(0)=g(\langle\rangle)$,

\item[(8)] $g^{\bold b}_\eta(n+1)=g(\eta \upharpoonright_L n)$\footnote{Here we are using a modified version of the twofold $\lambda$-black box $\bold b$, which can be easily obtained from the original one.},

\item[(9)]  $\eta(n, 1) < \eta(n, 2)$ and $f(\eta \upharpoonright_L n)=f(\eta \upharpoonright_R n)$   for all   $1\leq n < \omega$.
\end{enumerate}
Applying $h$ to the both sides of (8), one has $$a^{\bold c}_{\eta, n}\stackrel{(3)}=h(g^{\bold b}_\eta(n+1))=h(g(\eta\rest_Ln))\stackrel{(+)}=a_{n},$$
thereby completing the proof.
\end{proof}

\begin{lemma}
\label{g10} Assume ${\bold c}\in \AP$ is full and let $h \in \Hom(G_{\bold c}, K)$ be unbounded. Then there is a sequence $$\langle a_n: n<\omega     \rangle \in{}^{\omega}\rng(h)$$
such that the following set of equations $\Gamma$ has no solution, not only in $G_{\bold c}$,
but in any
$G_{\bold d}$ with  $\bold c\leq  \bold d \in \AP$, where
\[
\Gamma:=\{z_n= a_n+ n! z_{n+1}: n< \omega       \}.
\]
\end{lemma}
\begin{proof}
We have two possibilities.
First, suppose for some prime number $p$, the group  $\Gamma_p(\rng(h))$ is infinite, and let $p$ be the first such prime number. Also,  let $p_n=p$ for all $n<\omega$.
Otherwise, we let
$$p_n \in \big\{p: \Gamma_p(\rng(h))\neq 0  \big\}$$
be a strictly increasing sequence of prime numbers. We refer this as a second possibility.

In the first part of the proof, we argue for both possibilities at the same time. Then, we consider each scenario separately.

Since $h$ is not bounded, we can find by induction on $n$, the pair $(H_n, a_n)$
such that:
\begin{enumerate}
\item[$(+)$]
\begin{enumerate}
\item $H_0=\Rang(h)$,

\item $H_{n}=a_n \mathbb{Z} \oplus H_{n+1}$,

\item $a_n$  has order $p_n^{\bold l_n}$, \item for $n=m+1$ we have $$(d_{n}):\quad\bold l_n > \bold l_m+ \left(\prod_{i=0}^{n+1}i!\right).$$
\end{enumerate}
\end{enumerate}
To see this, let $H_0:=\Rang(h)$ and let $a_0 \in \Gamma_{p_0}[\rng(h)]$ be any nonzero element. Now, suppose inductively that $n>0$ and we have defined $\langle H_i: i \leq n    \rangle$
and $\langle  a_i: i<n     \rangle$ satisfying the above items. We shall now  define $a_n$
and $H_{n+1}$. By our induction assumption,
\[
\Rang(h)=(\bigoplus_{i<n}a_i\mathbb{Z}) \oplus H_n.
\]
In particular, $H_n$ is torsion. Using Fact \ref{kapla2} (and also Fact \ref{bounded_exponent} in the second possibility case), we can find for some $\ell_n$ and an element $a_n$ such that
$a_n$  has order $p_n^{\bold l_n}$ and $a_n \mathbb{Z}$ is a direct summand of $H_n$. We may further suppose that $$\bold l_n > \bold l_m+ \left(\prod_{i=0}^{n+1}i!\right).$$ Since 
$(a_n)$ is a direct summand of $H_n$, there is an abelian group
$H_{n+1}$  so that $H_{n}=a_n \mathbb{Z} \oplus H_{n+1}$.

To prove that the sequence $\langle a_n: n<\omega     \rangle$
is as required, assume towards a contradiction that there is $\bold c\leq  \bold d \in \AP$ such that $\langle c_n: n<\omega     \rangle$
is a solution of $\Gamma$ in $G_{\bold d}$. So
\[
G_{\bold d}\models  \bigwedge\limits_{n<\omega}\big(  c_n=a_n + n! c_{n+1}\big)\quad(\ast)
\]
Since for each $n, a_n \in K,$ it  follows that
 $$G_{\bold d}/K \models \bigwedge\limits_{n<\omega} \big(c_n +K= n! c_{n+1}+K\big).$$

By Lemma \ref{g31}, $G_{\bold c}/K$ is reduced,
hence necessarily,
$$\bigwedge\limits_{n<\omega}  \big(c_n +K= 0+K\big).$$
In other words, $c_n \in K$ for all $n<\omega.$

We now show that for each $n,$
$$\left( \prod_{i < n}i!\right)c_n \in H_n \quad(\ast\ast)$$
This is true for $n=0$, because $c_0 \in K=H_0$. Suppose it holds for $n$. Then multiplying both sides of $(*)$ into
$\prod_{i < n}i!$ we get
\[
\left(\prod_{i < n}i!\right) c_n =\left( \prod_{i < n}i!\right) a_n +\left( \prod_{i < n+1}i!\right) c_{n+1}.
\]
Using the induction hypothesis and $(\star)(b)$ we get $$\left(\prod_{i < n+1}i!\right) c_{n+1} \in H_{n+1},$$ as requested.

By an easy induction, for each $n$ we have
\[
c_0= a_0+ \sum\limits_{\ell \leq n} \left(\prod_{i=1}^{\ell}i!\right) a_\ell+ \left(\prod_{i=1}^{n}i!\right) c_{n+1} \quad(\ast\ast\ast)_n
\]
Indeed this is true for $n=0$, as $c_0=a_0+c_1$. Suppose it holds for $n$, then using $(\ast)$ and the induction hypothesis
\begin{equation*}
\begin{array}{clcr}
c_0 &= a_0+ \sum\limits_{\ell \leq n} \left(\prod_{i=1}^{\ell}i!\right) a_\ell+ \left(\prod_{i=1}^{n}i!\right)c_{n+1}
\\
&= a_0+ \sum\limits_{\ell \leq n} \left(\prod_{i=1}^{\ell}i!\right) a_\ell+ \left(\prod_{i=1}^{n}i!\right)\left(a_{n+1}+(n+1)!c_{n+2}\right)\\
&= a_0+ \sum\limits_{\ell \leq n+1} \left(\prod_{i=1}^{\ell}i!\right) a_\ell+ \left(\prod_{i=1}^{n+1}i!\right) c_{n+2}.
\end{array}
\end{equation*}
We are now ready to complete the proof. Let $m(*)$ be the order of $c_0$.

Now, we consider each case separately:
\\
{\bf Case 1. $p_n=p$ for all $n$:} Let $t$ be an integer such that   $$m(*)=tp^{\ell(*)} > 1,$$ where $\ell(*)\geq 0$
and $(p, t)=1$, i.e., $p$ does not divide $t$. Let
$k$ be the least natural number such that $\bold l_k > \ell(*)$.
By multiplying both sides of $(\ast\ast\ast)_{k+1}$
into $tp^{\bold l_k}$, we get to
\[
tp^{\bold l_k} c_0=tp^{\bold l_k} a_0 + tp^{\bold l_k} \sum\limits_{\ell \leq k+1} \left(\prod_{i=1}^{\ell}i!\right) a_\ell+t p^{\bold l_k}\left(\prod_{i=1}^{k+1}i!\right) c_{k+2}.
\]
Since the sequence $\langle \bold l_{\ell}: \ell \leq k      \rangle$
is increasing, we have $p^{\bold{l}_k}a_\ell=0$ for all $\ell \leq k.$ Consequently, 
\[
0=tp^{\bold l_k} \left(\prod_{i=1}^{k+1}i!\right) a_{k+1} +t p^{\bold l_k}\left(\prod_{i=1}^{k+1}i!\right) c_{k+2}\quad(\dagger)
\]
According  to $(+)_b$, we know $a_{k+1} \mathbb{Z} \cap H_{k+2}=0$, and by using $(\ast\ast)$ along with  $(\dagger)$ we get that
\[
t p^{\bold l_k} \left(\prod_{i=1}^{k+1}i!\right) a_{k+1}=0.
\]
Recall that the order of $a_{k+1}$ is a power of $p$. We apply this along with  the equality $(p, t)=1$ to get that
$$p^{\bold l_k} \left(\prod_{i=1}^{k+1}i!\right) a_{k+1}=0.$$
Moreover, $$p^{\bold l_{k+1}}=\ord(a_{k+1})\leq p^{\bold l_k} \left(\prod_{i=1}^{k+1}i!\right)\leq p^{\bold l_k +\left(\prod_{i=1}^{k+1}i!\right)}.$$ Taking $\log_p(-)$ from both sides, 
we have $\bold l_{k+1} \leq \bold l_k +\left(\prod_{i=1}^{k+1}i!\right)$. But, this contradicts   $(d_{\bold l_{k+1}})$.
The result follows.

Thereby, without loss of generality we deal with:

{\bf Case 2. Otherwise:} Then the sequence $\langle  p_n: n<\omega     \rangle$
is strictly increasing. Let $k$ be the least integer such that $$p_{k+1} >  m(*)\times \left(\prod_{i=1}^{k+1}i!\right)\quad(\dagger\dagger)$$By multiplying both sides of $(\ast\ast\ast)_{k+1}$
into $m(*)\times\left( \prod_{i=1}^{k}p_i^{\bold l_i}\right)$ we get

\begin{equation*}
\begin{array}{clcr}
0&=m(*) \times\left( \prod_{i=1}^{k}p_i^{\bold l_i}\right) c_0 \\&= m(*) \times\left( \prod_{i=1}^{k}p_i^{\bold l_i}\right) a_0
+
 m(*)\times\left( \prod_{i=1}^{k}p_i^{\bold l_i}\right) \sum\limits_{\ell \leq k+1} \left(\prod_{i=1}^{\ell}i!\right) a_\ell \\
&\quad  +m(*)\times\left( \prod_{i=1}^{k}p_i^{\bold l_i}\right)\left(\prod_{i=1}^{k+1}i!\right) c_{k+2}.
\end{array}
\end{equation*}
We have that
$m(*) \times\left( \prod_{i=1}^{k}p_i^{\bold l_i}\right) a_0=0$ and

$$m(*)\times\left( \prod_{i=1}^{k}p_i^{\bold l_i}\right)\left(\prod_{i=1}^{\ell}i!\right) a_\ell=0,$$for all $\ell \leq k,$ thus
\[
0=m(*)\times\left( \prod_{i=1}^{k}p_i^{\bold l_i}\right)   \left(\prod_{i=1}^{k+1}i!\right) a_{k+1} +m(*)\times\left( \prod_{i=1}^{k}p_i^{\bold l_i}\right)\left(\prod_{i=1}^{k+1}i!\right) c_{k+2}.
\]
Again, according  to $(+)_b$, we know $a_{k+1} \mathbb{Z} \cap H_{k+2}=0$, and by using $(\ast\ast)$ along with  the previous formula, we lead to the following vanishing formula
\[
m(*)\times\left( \prod_{i=1}^{k}p_i^{\bold l_i}\right)   \left(\prod_{i=1}^{k+1}i!\right) a_{k+1}=0.
\]
As the order of $a_{k+1}$ is a power of $p_{k+1}$ and it is different from all $p_\ell$'s, for $\ell\leq k,$ we have
 \[
m(*)\times  \left(\prod_{i=1}^{k+1}i!\right) a_{k+1}=0.
\]
So, $$p_{k+1}<p^{\bold l_{k+1}}_{k+1}=\ord(a_{k+1})\leq   m(*)\times \left(\prod_{i=1}^{k+1}i!\right).$$ But this contradicts   $(\dagger\dagger)$.
The result follows.
\end{proof}

To prove the endo-rigidity property, we first deal with the following special case, and then we reduce things to this situation:
\begin{lemma}
\label{g11}
Let $\bold c\in \AP$ be full.  Then every $h \in \Hom(G_{\bold c},K)$
is bounded.
\end{lemma}
\begin{proof}
Towards a contradiction assume $h \in \Hom(G_{\bold c},K)$
is not bounded.
In view of Lemma \ref{g10}, this implies that there is a sequence $$\langle a_n: n<\omega     \rangle \in{}^{\omega}\rng(h)$$
 such that
the set of equations
\[
\Gamma:=\{z_n= a_n+ n! z_{n+1}: n< \omega       \}
\]
has no solutions in $G_{\bold c}.$
Let $\gamma=|K|$,  and  define $f:\Lambda_{<\omega}[\lambda]\to\gamma$
such that $$f(\eta)=f(\nu)\Longleftrightarrow
h(x_\eta)=h(x_\nu)\quad(\ast)$$
Since $a_n \in\rng(h)$ there is $b_n$ such that $$\forall n<\omega,~a_n=h(b_n)\quad \quad(+)$$
As $\bold c$ is full, we can find some $\eta$ such that

\begin{enumerate}
	\item $f(\eta\rest_Ln)=f(\eta\rest_Rn)$,
	
	\item $a^{\bold c}_{\eta, n}=b_n$ for each $n$.
	
\end{enumerate}
Let us combining $(\ast)$ and (1). This yields  that
$$\forall n<\omega,~h(x_{\eta\rest_Ln})=h(x_{\eta\rest_Rn}) \quad\quad(\dagger).$$
Moreover, by applying $h$ to the both sides of the equation
$$y_{\eta, n}=a^{\bold c}_{\eta, n}+(n!)y_{\eta, n+1}+(x_{\eta\rest_Ln}-x_{\eta\rest_Rn}),$$
we lead to the following equation:

\begin{equation*}
\begin{array}{clcr}
h(y_{\eta, n})&= h(a^{\bold c}_{\eta, n})+ n! h(y_{\eta, n+1})+\big(h(x_{\eta\rest_Ln})-h(x_{\eta\rest_Rn})\big)
\\&\stackrel{(2)}= h(b_n)+ n! h(y_{\eta, n+1})+\big(h(x_{\eta\rest_Ln})-h(x_{\eta\rest_Rn})\big)\\
&\stackrel{(\dagger)}= h(b_n)+(n!)h(y_{\eta, n+1})\\
&\stackrel{(+)}= a_{ n}+(n!)h(y_{\eta, n+1}).

\end{array}
\end{equation*}
In other words, $h(y_{\eta, n})$ is a solution
for
\[
\Gamma=\{z_n= a_n+ n! z_{n+1}: n< \omega       \}.
\]
This is a contradiction with the choice of the sequence
$\langle a_n: n < \omega \rangle.$
\end{proof}

\begin{notation}\label{cof}Suppose $\bold c\in \AP$.
	For each $n<\omega$, we define $$G_n:=\frac{G_\bold c}{ K+\big(\prod_{i=1}^n i!\big)G_\bold c}.$$
	Also, the notation $\pi_n$ stands for the natural projection  $G_\bold c \twoheadrightarrow G_n.$
\end{notation}

\begin{fact}\label{fcof}
	Adopt the above notation, let $n<\omega$ and   $g\in G_\bold c$.
	\begin{itemize}
		\item[(a)] The abelian group $G_n$ is a torsion abelian group  with the following minimal  generating set $$\{x_\rho:\rho\in\Lambda_{\bold c, < \omega}\}
		\cup\{y_{\eta,k}:\eta\in\Lambda_{\bold c,\omega} \emph { and } k\geq n+2\}.$$ \item[(b)] Similar to Definition \ref{circ}, we can define $\supp_\circ(\pi_n(g))$ with respect to generating set presented in clause $(a)$.
		\item[(c)] According to its definition, it is easy to see that $\supp_\circ(\pi_n(g))\subseteq\supp_\circ(g)$.
		\item[(d)] Recall from Lemma \ref{g31} that $G_c/ K$ is reduced.  This in turns  gives us an integer $m_n>n$ such that $\supp_\circ(g)\subseteq\supp_\circ(\pi_{m_n}(g))$.
	\end{itemize}
\end{fact}

\begin{proof}
	This is straightforward.
\end{proof}

\begin{lemma}
	\label{g12}
	Suppose $\bold c\in \AP$ is full and $h \in \End(G_{\bold c})$. Then for some countable $\Lambda_h \subseteq \Omega_{\bold c}$
	we have:
	\begin{center}
		$r \in R,~\nu \in \Omega_{\bold c} \setminus \Lambda_h \implies \supp_\circ(h(rx_\nu)) \subseteq \{\nu\} \cup \Lambda_h.$
	\end{center}
\end{lemma}
\begin{proof}
	Towards contradiction assume $h \in \End(G_{\bold c})$ but there is no  $\Lambda_h$ as promised.
We define a sequence
	$$\big\langle (\eta_i, Y_i,\nu_i, r_i): i < \omega_1\big \rangle,$$
	by induction on $i<\omega_1$, such that
	\begin{enumerate}
		\item[($\ast$)]
		\begin{enumerate}
			\item $\eta_i\in \Omega_{\bold c}$ and $r_i \in R \setminus \{0\},$
			
			\item $Y_i = \bigcup\{\supp_\circ(h(r_j x_{\eta_j})): j<i       \} \cup \{ \eta_j: j<i\},$
			
			\item $\nu_i \in \supp_\circ(h(r_ix_{\eta_i}))$ but $\nu_i \neq \eta_i$, $\nu_i \notin Y_i.$
			
		\end{enumerate}
	\end{enumerate}
	To this end,
	 suppose that $i<\omega_1$ and we have defined
	$\langle (\eta_j, Y_j,\nu_j, r_j):j< i  \rangle$. Set $$Y_i = \bigcup\{\supp_\circ(h(r_jx_{\eta_j})): j<i       \} \cup \{ \eta_j: j<i\}.$$ Following its definition, we know $Y_i$ is  at most countable. Thus, due to our assumption, we can find some $\eta_i \in \Omega_{\bold c} \setminus Y_i$ and $r_i\in R\setminus\{0\}$
	such that $$\supp_\circ(h(r_ix_{\eta_i})) \nsubseteq  (\{\eta_i\} \cup Y_i).$$
	This allow us to define $\nu_i$, namely, it is enough to take $\nu_i$ be any element of  $\supp_\circ(h(r_ix_{\eta_i})) \setminus (\{\eta_i\} \cup Y_i)$.
	This completes the definition of
	$(\eta_i, Y_i,\nu_i, r_i)$.
	

Combining the facts
 $\nu_i \in \supp_\circ(h(r_ix_{\eta_i}))$ and
$\nu_i\notin(Y_i\cup\{\eta_i\})$ along with the finiteness of
$\supp_\circ(h(x_{\eta_i}))$
we are able to find a subset $W\subseteq \omega_1$ of cardinality $\omega_1$ such
that

	\begin{enumerate}
	\item[($\ast$)] If $i \neq j\in W$
	then $\nu_j\notin\supp_\circ(h(r_ix_{\eta_i}))$.
		\end{enumerate}
Without loss of generality we may and do assume that $W=\omega_1$. Let $a_i=r_ix_{\eta_i}$.
We can find $$f:\Lambda_{\bold c, <\omega}\longrightarrow |R|+\aleph_0 < \lambda$$ such that if $b\in G_{\bold c}$\footnote{Recall we have chosen $$b-\sum_{\ell<n_b} r_{b,\ell}y_{\eta_{b,\ell}, m_{b,\ell}}\in\sum_{\rho\in\Lambda_{\bold c, < \omega}}Rx_\rho +K.$$}, then from $f(b)$ we can compute
$$\big\langle n_b,\{(\ell,m_{b, \ell}, r_{b,\ell}):\ell<n_b\}\big \rangle.$$

Recall that $\bold c$ is full, and that $\range(f)$ has size less than $\lambda$. From these, there is some $\eta \in \Lambda_{\bold c, \omega}$ furnished with the the following two properties:
\begin{enumerate}
	\item $f(\eta \restriction_L n)=f(\eta \restriction_R n),$ for $n<\omega$,
	
	\item  $a^{\bold c}_{\eta, n}=a_n$ for all $n<\omega$.
\end{enumerate}
Now, we bring the following claim.

 $\bold{Claim}$:  $\nu_i\in\supp_0(h(y_{\eta_0}))\quad\forall i<\omega.$

Note that this will give us the desired  contradiction, as $\supp_0(h(y_{\eta_0}))$ is finite. Now we turn to the proof of the claim.
\begin{proof}[Proof of Claim]
By Lemma \ref{g21} we first observe that:
	\begin{equation*}
\begin{array}{clcr}
y_{\eta, 0}&= \sum\limits_{i=0} ^n \big(\prod_{j<i}j!\big)r_i x_{\eta_i}+\big(\prod_{i=1}^{n}i!\big) y_{\eta, n+1}+ \sum\limits_{i=0}^n \big(\prod_{j<i}j!\big)(x_{\eta \restriction_L i}-x_{\eta \restriction_R i}).
\end{array}
\end{equation*}

Let $\ell$ be any integer.  We are going to use the notation presented in Notation \ref{cof} for $n=m_\ell$.
Applying $\pi_n h(-)$ to it, yields that

\begin{equation*}
\begin{array}{clcr}
(3) \quad \pi_n( h(y_{\eta, 0}))&= \sum\limits_{i=0} ^n \big(\prod_{j<i}j!\big)\pi_nh(r_i x_{\eta_i})+(\prod_{i=1}^{n}i!) \pi_nh(y_{\eta, n+1}) \\&\quad+\sum\limits_{i=0}^n \big(\prod_{j<i}j!\big)\pi_nh(x_{\eta \restriction_L i}-x_{\eta \restriction_R i})
\\&\stackrel{}= \sum\limits_{i=0} ^n \big(\prod_{j<i}j!\big)\pi_nh(r_i x_{\eta_i})+\sum\limits_{i=0}^n \big(\prod_{j<i}j!\big)\pi_nh(x_{\eta \restriction_L i}-x_{\eta \restriction_R i}),\\

\end{array}
\end{equation*}
where the last equality follows by Definition \ref{cof}.
Now, we recall from the construction $(\ast)$ that:
\begin{enumerate}
	\item[(3.1)]  $\nu_i \in \supp_\circ(h(r_ix_{\eta_i}))$,	\item[(3.2)] $\nu_i \neq \eta_i$ and $\nu_i \notin Y_i.$

\end{enumerate}

Thanks to Fact \ref{fcof}(d)
we have
\begin{enumerate}
	\item[(4)] $\nu_i \in \supp_\circ(\pi_{n}h(r_ix_{\eta_i}))$.
\end{enumerate}
By clause (1) above, $\supp_\circ(h(x_{\eta \restriction_L i}-x_{\eta \restriction_R i}))=\emptyset$.
In view of Fact \ref{fcof}(c),
we deduce that

\begin{enumerate}
	\item[(5)] $\supp_\circ\big(\pi_n(h(x_{\eta \restriction_L i}-x_{\eta \restriction_R i}))\big)=\emptyset$.
\end{enumerate}
First, we plug  items (4) and (5) in the clause (3), then we use $(*)$. These enable  us
to observe that

\begin{equation*}
\begin{array}{clcr}
\nu_i& \in \supp_\circ\bigg( \sum\limits_{i=0} ^n \big(\prod_{j<i}j!\big)\pi_nh(r_i x_{\eta_i})+\sum\limits_{i=0}^n \big(\prod_{j<i}j!\big)\pi_nh(x_{\eta \restriction_L i}-x_{\eta \restriction_R i})\bigg) \\&= \supp_\circ(\pi_nh(y_{\eta, 0})).
\end{array}
\end{equation*}

Another use of Fact \ref{fcof}(c),
shows that $
\nu_i \in   \supp_\circ(h(y_{\eta, 0})).
$ This completes the proof of the claim.
\end{proof}
The lemma follows.
 \end{proof}

The following lemma can be proved easily.
\begin{lemma}
\label{g121}
Let $\bold c\in \AP$ be full and $h \in \End(G_{\bold c})$. Let
 $Y_0\subseteq \Omega_{\bold c}$ be the downward closure  of $\Lambda_h$, where $\Lambda_h$ is as in Lemma \ref{g12} and set
 $$K^+:=K+\sum_{\rho\in Y_0\cap\Lambda_{\bold c, < \omega}}Rx_\rho+\sum_{\rho\in Y_0\cap\Lambda_{\bold c, \omega}n<\omega}Ry_{\rho,n}.$$
 If $b\in G_{\bold c}$, then there are choices
	\begin{enumerate}
		\item[$\bullet$] $\bar r_b:=\langle r^2_{b,\rho}:\rho\in\Lambda _b \rangle $, and
		\item[$\bullet$] $\Lambda_{\bold b}\subseteq \Lambda_{\bold c, < \omega}\setminus Y_0$ finite
	\end{enumerate}
such that$$b-\sum_{\rho\in \Lambda_{\bold b}}r^2_{b,\rho} x_\rho\in K^+.$$
\end{lemma}
\begin{proof}
This is straightforward.
\end{proof}
\begin{hypothesis}
For the rest of this section, we fix a well-ordering $\prec$ of the large enough part of the universe, and for each:
	\begin{enumerate}
	\item[$\bullet$] $\bold c\in \AP$ which is full,
	\item[$\bullet$] $h \in \End(G_{\bold c})$, and
	\item[$\bullet$]$b\in G_{\bold c}$,
\end{enumerate}
we
let  $\bar r_b:=\langle r^2_{b,\rho}:\rho\in\Lambda _b \rangle  $ be the $\prec$-least sequence satisfying the conclusions of Lemma \ref{g121}.
\end{hypothesis}

\begin{notation}\label{notl} Suppose $\bold c \in \AP$ and $\Lambda \subseteq \Lambda_{\bold c}$. By  $G_{\bold c, \Lambda}$ we mean $$G_{\bold c, \Lambda}:=G_{ \Lambda}:= \bigg\langle \big\{rx_\nu, ry_{\eta, n}: r \in R, \nu \in  \Lambda_{<\omega}, \eta \in  \Lambda_{\omega} \emph{ and } n<\omega  \big \}              \bigg\rangle.$$
\end{notation}
We have the following observation, but as we do not use it, we leave its proof.
\begin{observation}
	\label{g9}
	Suppose $\Lambda \subseteq \Lambda[\lambda]$ is downward closed. Then
	$G_{\bold c, \Lambda}$ is a pure subgroup of $G_{\bold c}$.
\end{observation}

\begin{lemma}\label{findm}
Let $\bold c\in \AP$ be full, and $h \in \End(G_{\bold c})$. Then for some countable  $\Lambda_h\subseteq \Lambda[\lambda]$
we have:
\[
r \in R, \nu \in \Omega_{\bold c} \setminus \Lambda_h\implies
h(r x_\nu) \in  G_{\bold c, \Lambda_h\cup\{\nu\}}+K.
\]
\end{lemma}

\begin{proof}

Suppose  on the way of contradiction that the lemma fails. Let $Y_0$ be as
	 Lemma \ref{g121}. We define a sequence
$$\langle (Y_i,\nu_i, \rho_i,r_i): i < \omega_1 \rangle,$$
by induction on $i<\omega_1$, such that	

\begin{enumerate}
	\item[($\natural$)]
	\begin{enumerate}
		\item  $r_i \in R \setminus \{0\},$
		
		\item $Y_i = \bigcup\{\supp(h(r_j x_{\nu_j})): j<i       \} \cup \{ \rho_j: j<i\}\cup Y_0,$
		
		\item $\nu_i \in \Omega_{\bold c}\setminus Y_i$,
		
	 \item $h(r_i\nu_i)\notin  G_{\bold c, Y_i\cup\{\nu_i\}}+K,$  \item
	 let $b_i:=h(r_i\nu_i)$, and let $\bar r_{b_i}:=\langle r^2_{b_i,\rho}:\rho\in\Lambda _i \rangle $ be as  Lemma \ref{g121} applied to $b_i$.
	 Then $\rho_i\in \Lambda_i\setminus(Y_i\cup\{\nu_i\}),$ and even
\[
r^2_{b_i,\rho_i} x_{\rho_i} \notin  G_{\bold c, Y_i\cup\{\nu_i\}}+K.
\]
	\end{enumerate}\end{enumerate}
  To construct this, suppose $i< \omega$ and we have constructed the sequence up to $i$. Now, $(\natural)_b$ gives the definition of $Y_i$.
Since we assume that the lemma fails, there is an $ r_i \in R$ and $ \nu_i \in \Omega_{\bold c} \setminus Y_i$ such that
	  $h(r_i x_{\nu_i}) \notin  G_{\bold c, \Lambda_h\cup\{\nu\}}+K.$
	Now, we define $b_i:=h(r_i\nu_i)$.
	 Thanks to Lemma \ref{g121},
	there is a finite set $\Lambda_i \subseteq \Lambda_{\bold c, <\omega}\setminus Y_i$ and a  sequence $\langle r^2_{b_i,\rho}: \rho \in \Lambda_i \rangle$ such that $$b_i-\sum_{\rho\in \Lambda_{i}}r^2_{b_i,\rho} x_\rho\in K^+.$$
	 As $b_i\notin  G_{\bold c, Y_i\cup\{\nu_i\}}+K$
and due to the following containment $$b_i-\sum_{\rho\in \Lambda_{i}}r^2_{b_i,\rho} x_\rho\in K^+\subseteq  G_{\bold c, Y_i\cup\{\nu_i\}}+K,$$
there is $\rho_i\in \Lambda_{i}$ such that $\rho_i\notin (Y_i\cup\{\nu_i\})$, and indeed $$r^2_{b_i,\rho_i} x_{\rho_i} \notin  G_{\bold c, Y_i\cup\{\nu_i\}}+K.$$ This completes the proof of construction.
By shrinking the sequence, we may and do assume in addition that
\begin{enumerate}
\item[$\bullet$] for all $i \neq j < \omega_1, \rho_j \notin \Lambda_i$.
\end{enumerate}

Let $a_n:=r_nx_{\nu_n}$ and define $$f:\Lambda_{\bold c,<\omega}\to \mid R\mid+\mid K\mid+ \aleph_{0}<\lambda
$$be such that for any $\rho\in\Lambda_{\bold c,<\omega}$, $f(\rho)$ codes
\begin{itemize}
		\item[$\bullet$] $\langle r^2_{b,\rho}:\rho\in\Lambda _b \rangle $, and
	\item[$\bullet$] $b-\sum_{\nu\in \Lambda_{i}}r^2_{b,\nu} x_\nu,$
\end{itemize}
where $b:=h(x_{\rho})$. To see such a function $f$ exists, first we define:

	\begin{enumerate}
		\item[$\bullet$]  $f_1: R^{<\omega} \times K^+ \to  \mid R\mid+\mid K\mid+ \aleph_{0}$ is a bijection, and
		
		\item[$\bullet$]  $f_2: \Lambda_{\bold c,<\omega}\to R^{<\omega} \times K^+$ is defined as
		\[
		f_2(b)= \big(\langle r^2_{b,\rho}:\rho\in\Lambda _b \rangle, b-\sum_{\nu\in \Lambda_{i}}r^2_{b,\nu} x_\nu\big).
		\]
\end{enumerate}
Then, we set $f:=f_1 \circ f_2.$
Suppose $\rho_1,\rho_2\in\Lambda_{\bold c,<\omega}$
are such that $f(\rho_1)=f(\rho_2)$. We claim that $h(x_{\rho_1})=h(x_{\rho_2})$. To see this, it is enough to apply $f(\rho_1)=f(\rho_2
)$, and conclude that
\begin{itemize}
	\item[$(1)$] $\langle r^2_{b_1,\nu}:\nu\in\Lambda _{b_1} \rangle =\langle r^2_{b_2,\nu}:\nu\in\Lambda _{b_2} \rangle$
	\item[$(2)$] $b_1-\sum_{\nu\in \Lambda_{b_1}}r^2_{b,\nu} x_\nu=b_2-\sum_{\nu\in \Lambda_{b_2}}r^2_{b,\nu} x_\nu$,
\end{itemize}
	where $b_i=h(x_{\rho_i})$. But, then we have
	\begin{equation*}
	\begin{array}{clcr}
b_1&=b_1-\sum_{\nu\in \Lambda_{b_1}}r^2_{b,\nu} x_\nu+(\sum_{\nu\in \Lambda_{b_1}}r^2_{b,\nu} x_\nu)\\&\stackrel{(2)}= b_2-\sum_{\nu\in \Lambda_{b_2}}r^2_{b,\nu} x_\nu+(\sum_{\nu\in \Lambda_{b_2}}r^2_{b,\nu} x_\nu)
	\\&=b_2,
\end{array}
	\end{equation*}i.e.,
$h(x_{\rho_1})=h(x_{\rho_2})$.

Since $\bold c$ is full, and in the light of Definition \ref{g6}(b),  we are able to find an $\eta\in\Lambda_{\bold c, \omega}$ such that
\begin{itemize}
	\item[$(3)$] $a_n=a^{\bold c}_{\eta,n}$, and
	\item[$(4)$]  $f({\eta\rest_Ln})=f({\eta\rest_Rn})$,
\end{itemize} for all $n<\omega$.
 Thanks to the previous paragraph and clause (4) we deduce $$h(x_{\eta\rest_Ln})=h(x_{\eta\rest_Rn})\quad(\sharp)$$
By applying $h$ to the both sides of the following equation
 \begin{equation*}
\begin{array}{clcr}
y_{\eta, 0}&= \sum\limits_{i=0} ^n \big(\prod_{j<i}j!\big)r_ix_{\nu_i}+\big(\prod_{i=1}^{n}i!\big) y_{\eta, n+1}+ \sum\limits_{i=0}^n \big(\prod_{j<i}j!\big)(x_{\eta \restriction_L i}-x_{\eta \restriction_R i}),
\end{array}
\end{equation*}
we get

\begin{equation*}
\begin{array}{clcr}
h(y_{\eta, 0})&= \sum\limits_{i=0} ^n \big(\prod_{j<i}j!\big)h(r_ix_{\nu_i})+ \big(\prod_{i=1}^{n}i!\big) h(y_{\eta, n+1})\\&~\quad~+\big(\prod_{j<i}j!\big)\big(h(x_{\eta\rest_Ln})-h(x_{\eta\rest_Rn})\big)
\\&\stackrel{(\sharp)}= \sum\limits_{i=0} ^n \big(\prod_{j<i}j!\big)h(r_ix_{\nu_i})+ \big(\prod_{i=1}^{n}i!\big) h(y_{\eta, n+1})\quad(+)\\
\end{array}
\end{equation*}
For each $i<\omega_1$ let $b_i=h(r_i x_{\nu_i})$.
Let also $b=h(y_{\eta, 0})$ and let $\Lambda_{b}$ be as in Lemma \ref{g121}. As $\Lambda_b$ is finite, for some  large enough $n$, we have
$$\{\rho_i:i<n\}\setminus \Lambda_b\neq \emptyset.$$
Let $i<n$ be such that $\rho_i\notin \Lambda_b$. Here, we apply the argument presented in items (3)-(5) from Lemma \ref{g12} to the displayed formula $(+)$. So, on the one hand, it turns out that $$\rho_i\in \Lambda_i \subseteq\Lambda_b.$$
On the other hand by the choice of $i$, $\rho_i \notin \Lambda_b.$ This is a contraction that we searched for it.
\end{proof}

\begin{lemma}
	\label{g13}
	Let $\bold c\in \AP$ be full, and $h \in \End(G_{\bold c})$. Then for some $m_* \in R$ and some countable  $\Lambda_h = \cl(\Lambda_h) \subseteq \Lambda[\lambda]$
	we have:
	\[
	r \in R, \nu \in \Omega_{\bold c} \setminus \Lambda_h\implies
	h(r x_\nu)-m_* x_\nu \in  G_{\Lambda_h}+K.
	\]
\end{lemma}
\begin{proof}
	In view of  Lemma \ref{findm},  there is some countable downward closed subset $\Lambda $ of $ \Lambda_{\bold c}$ such that for every $r \in R$ and
	$\eta \in \Omega_{\bold c} \setminus \Lambda$, we have  $h(rx_\eta) \in G_{\Lambda \cup \{\nu\}}+K.$  Thus, for such $r$ and $\eta$, there are $m^r_\eta \in R$ and $b^r_\eta$ satisfying the following two   properties:
	\begin{itemize}
		\item $h(rx_\eta)=m^r_\eta x_{\eta}+b^r_\eta$,
		
		\item $b^r_\eta \in  G_{\Lambda}+K.$
	\end{itemize}
	Suppose on the way of contradiction that the desired conclusion fails.
	By induction on $i<\omega_1$ we define a sequence
	$$\big\langle Y_i, r_{i, 1}, r_{i, 2}, \eta_{i, 1}, \eta_{i, 2}:i<\omega_1 \big\rangle$$
	such that:
	\begin{enumerate}
		\item[($\dag$)]
		\begin{enumerate}
			\item $Y_i=\Lambda \cup \{\eta_{j, \ell}: j<i, \ell \in \{1, 2\}    \},$
			
			\item $r_{i, 1}, r_{i, 2} \in R \setminus \{0\}$,
			\item $\eta_{i, \ell} \in\Omega_{\bold c} \setminus Y_i$, for $\ell \in \{1, 2\},$
			
			\item $m^{r_{i, 1}}_{\eta_{i, 1}} \neq m^{r_{i, 2}}_{\eta_{i, 2}}$.
		\end{enumerate}
	\end{enumerate}
	The construction is easy, but we elaborate. Let us start with the case $i=0$. We  set $Y_0=\Lambda$
	and then choose $r_{0, 1}, r_{0, 2} \in  R \setminus \{0\}$ and $\eta_{0, 1}, \eta_{0, 2} \in \Lambda_{<\omega}[\lambda] \setminus \Lambda_h$ such that
	$m^{r_{0, 1}}_{\eta_{0, 1}} \neq m^{r_{0, 2}}_{\eta_{0, 2}}$. Now suppose $i<\omega_1$ and we have define the sequence for all $j<i$. Define $Y_i$ as in clause
	$(\dag)$(a). By our assumption, we can find
	\begin{enumerate}
	\item[i)]	$r_{i, 1}, r_{i, 2} \in R \setminus \{0\}$ and
	 \item[ii)] $\eta_{i, 1}, \eta_{i, 2} \in \Omega_{\bold c} \setminus Y_i$,	\end{enumerate}
	so that  $m^{r_{i, 1}}_{\eta_{i, 1}} \neq m^{r_{i, 2}}_{\eta_{i, 2}}$. This completes the induction construction.

Let $$f:\Lambda_{\bold c,<\omega}\to \mid R\mid+\mid K \mid+\aleph_{0}<\lambda
$$
be such that if
$r \in R$ and $\eta \in \Omega_{\bold c}$,
 then  $f(rx_\eta)$ is defined in a way that  one can compute $m^r_\eta$ and $b^r_\eta.$
 Again we can define $f$ as $$f=f_1 \circ f_2 \circ f_3,$$ where:

 	\begin{itemize}
 	\item $f_1: R \times (G_\Lambda+K) \to \mid R\mid+\mid K \mid+\aleph_{0}$ is a bijection,
 	
 \item	$f_2: R \times \Lambda_{\bold c,<\omega} \to R \times (G_\Lambda+K)$
 	is defined as
 $
 	f_2(r, \eta)=(m^r_\eta, b^r_\eta),
 $

 \item	 $f_3: \Lambda_{\bold c,<\omega} \to R \times \Lambda_{\bold c,<\omega}$ is
 	a bijection.

 \end{itemize}

	For each $n<\omega$, we set
	$$a_n:=r_{n,1}x_{\eta_{n, 1}} -r_{n,2}x_{\eta_{n, 2}} .$$
	Applying $h$ to it yields:
	
	 $$h(a_n)=m^{r_{n, 1}}_{\eta_{n, 1}}x_{\eta_{n, 1}} -m^{r_{n, 2}}_{\eta_{n, 2}}x_{\eta_{n, 2}}+b_n \quad(+),$$
where $b_n:=b^{r_{n, 1}}_{\eta_{n, 1}}-b^{r_{n, 1}}_{\eta_{n, 1}}$. Since $\bold c$ is full, there is an $\eta\in\Lambda_{\bold c, \omega}$ such that
	\begin{itemize}
		\item[$(1)$] $a_n=a^{\bold c}_{\eta,n}$, and
		\item[$(2)$]  $f({\eta\rest_Ln})=f({\eta\rest_Rn})$
	\end{itemize}for all $n<\omega$. By clause (2)
we deduce:
	\begin{itemize}
	\item[$(3)$]   $\supp_\circ (h(x_{\eta\rest_Ln}-x_{\eta\rest_Rn}))=\emptyset$ for all $n<\omega$.
\end{itemize}
	Applying $h$ to $$y_{\eta, 0}= \sum\limits_{i=0} ^n  a_i+\big(\prod_{i=1}^{n}i!\big) y_{\eta, n+1}+\sum\limits_{i=0}^n \big(\prod_{j<i}j!\big)(x_{\eta \restriction_L i}-x_{\eta \restriction_R i}),$$  yields that
	\begin{equation*}
	\begin{array}{clcr}(\natural)\quad
h(y_{\eta, 0})&= \sum\limits_{i=0} ^n  h(a_i)+\big(\prod_{i=1}^{n}i!\big) h(y_{\eta, n+1})+\big(\prod_{j<i}j!\big)\big(h(x_{\eta\rest_Ln})-h(x_{\eta\rest_Rn})\big)
	\\&\stackrel{(3)}=\sum\limits_{i=0} ^n  h(a_i)+\big(\prod_{i=1}^{n}i!\big) h(y_{\eta, n+1})
	\\
	&\stackrel{(+)}= \sum\limits_{i=0} ^n \big(m^{r_{n, 1}}_{\eta_{n, 1}}x_{\eta_{n, 1}} -m^{r_{n, 2}}_{\eta_{n, 2}}x_{\eta_{n, 2}}+b_n   \big)+   \big(\prod_{i=1}^{n}i!\big) h(y_{\eta, n+1}) .
	\end{array}
	\end{equation*}
Let $n<\omega$ be  large enough.
Here, we are going to apply the arguments taken from (3)-(5) in Lemma \ref{g12} to the displayed formula $(\natural)$.	Then, it turns out that
\begin{itemize}
	\item[$(4)$]  $\supp_\circ(h(y_{\eta, 0}))\supseteq\supp_\circ(h(a_n))$, and
	\item[$(5)$]  $\supp_\circ(h(a_n))\cap \{\eta_{n, 1},\eta_{n, 2}\} \neq \emptyset$.
\end{itemize}Without loss of generality, let us assume that for each $n<\omega,$ $\eta_{n, 1}\in\supp_\circ((h(a_n))$.
So, $$\{\eta_{n, 1}:n<\omega\}\subseteq\supp_\circ(h(y_{\eta, 0})),$$ which is   infinite. This is a contraction.
\end{proof}

\begin{lemma}
	\label{t11} Assume $\Lambda=\cl(\Lambda) \subseteq \Lambda_{\bold c}$ is countable and $h \in \Hom(G_{\bold c}, G_{\Lambda}+K)$. Then $h$
	is bounded.
\end{lemma}
\begin{proof}
	Towards a contradiction we assume that $h$ is unbounded. It follows from Lemma \ref{g11} that  $\range(h) \nsubseteq K$. Let $b_* \in \range(h) \setminus K.$ Then, for some  $d_* \in K$, a finite set $\Lambda_*$ and two sequences $\langle r_\eta \in R \setminus \{0\}: \eta \in \Lambda_* \rangle$
and $\langle m_\eta \in \omega: \eta \in \Lambda_*         \rangle$, we can represent $b_*$ as
		\[
	b_*=\sum\{r_\eta x_\eta: \eta \in \Lambda_* \cap \Lambda_{<\omega}      \}+ \sum\{r_\eta y_{\eta, m(\eta)}: \eta \in \Lambda_* \cap \Lambda_{\omega}      \}+d_*.
	\]
	Let
		\begin{enumerate}
			\item $J_0=G_{\Lambda}+K,$
			
			\item $J_1=J_0/K,$ which is torsion free.
			
			
		\end{enumerate}
	So, $b_* \in J_0$. Let $\pi: J_0 \to J_1$ be the natural map defined by the assignment $b\mapsto\pi(b):=b+K$. Since  $b_* \in \range(h) \setminus K, $ we have $\pi(b_*) \neq 0.$
Suppose on the way of contradiction that
for any
sequence $\langle e_n: n < \omega\rangle \in {}^{\omega}\mathbb{Z}$
the following system of equations
\[
\Gamma := \{y_n=n! y_{n+1} +e_n b_*: n<\omega  \}
\]
is solvable in $J_1$. Say for example, $\{y_n: n<\omega\}$ is such a solution.

Thanks to Lemma \ref{g3}(3)(a)
$\Lambda_{\bold c}$ is $\aleph_1$-free.
We combine this with Lemma \ref{g3}(3)(b)
to deduce that $M_{\bold c}$ is $\aleph_1$-free as an $R$-module.
Now, since
$J_1$ is countably generated,  we can find a solution  to
	\[
	\Gamma = \{y_n=n! y_{n+1} +e_n \bar{b}_*: n<\omega  \}
	\] in $R$.
Since $R$ is cotorsion-free, a such system of equations has no solution the ring.	So,
there is a sequence $\langle e_n: n < \omega\rangle \in {}^{\omega}\mathbb{Z}$
the following equations
\[
\Gamma = \{y_n=n! y_{n+1} +e_n b_*: n<\omega  \}
\]
is not solvable in $J_1$.

Let $a_* \in G_{\bold c}$ be such that $b_*=h(a_*)$. Let also $f: \Lambda_{\bold c, < \omega} \to \omega$ be such that for all $\nu, \rho \in \Lambda_{\bold c, < \omega},$
\[
f(\nu)=f(\rho) \Leftrightarrow \pi \circ h(x_\nu)=\pi \circ h(x_\rho).
\]
As $\bold c$ is full, there is some $\eta \in \Lambda_{\bold c, \omega}$
	such that:
	\begin{enumerate}
		\item[(3)] $a^{\bold c}_{\eta, n}=e_n a_*,$ for all $n<\omega$, and
		\item[(4)] $f(\eta \upharpoonright_L n)=f(\eta \upharpoonright_R n)$,  for $n<\omega.$
	\end{enumerate}
Thanks to (4), one has
$$\forall n< \omega,~ \pi \circ h(x_{\eta \upharpoonright_L n})=\pi \circ h(x_{\eta \upharpoonright_R n}) \quad\quad(+)$$
By applying $\pi \circ h$ into the equation
\[
y_{\eta, n} = a^{\bold c}_{\eta, n} + n! y_{\eta, n+1} + (x_{\eta \upharpoonright_L n} - x_{\eta \upharpoonright_R n}),
\]
and using clause (3) and $(+)$ we get
\[
\pi \circ h(y_{\eta, n}) = e_n \pi(b_*) + n! \pi \circ h(y_{\eta, n+1}).
\]
This clearly gives a contradiction, as then
\[
J_1 \models``y_n=n! y_{n+1} +e_n b_*'',
\]
where $y_n=\pi \circ h(y_{\eta, n})$.
\end{proof}

\begin{lemma}
	\label{g15}Let  $\bold c$ be full and $h \in \End(G_{\bold c})$. Then $\rng(h)$   is  bounded.
\end{lemma}
\begin{proof}
	Suppose not, it follows that for some countable $\Lambda=\cl(\Lambda) \subseteq \Lambda_{\bold c}$, $$h \restriction G \in \Hom(G, G_{\Lambda}+K)$$ is unbounded, where
	$G$ is the subgroup of $G_{\bold c}$ generated by $h^{-1}[G_{\Lambda}+K]$. This contradicts Lemma \ref{t11}.
\end{proof}

Now, we are ready to prove:

\begin{theorem}\label{main}Adopt the notation from Hypothesis
\ref{g1}.
	Then there is some $\bold c$	such that the abelian  group $G_{\bold c}$ is boundedly rigid. In particular, there is an abelian group $G$ equipped with the following properties
			\begin{enumerate}
		\item $\tor(G)=K$,
		
		\item $G$ is of size $\lambda$,
		
		\item the sequence $$0\longrightarrow   R_\bold c\longrightarrow \End(G)\longrightarrow \frac{\End(G)}{\BEnd(G)}\longrightarrow 0$$is exact.
		
	\end{enumerate}

\end{theorem}

\begin{proof}
	According to Lemma \ref{g7}, there is a full $\bold c \in \AP.$
	This allows us to apply Lemma \ref{g15}, and deduce that
	$G:=G_{\bold c}$ is boundedly rigid. By definition, this completes the proof.
\end{proof}	

\section{Co-Hopfian and boundedly endo-rigid abelian  groups}\label{5}
\bigskip

As stated in \cite{fuchs}, it
is difficult to construct an infinite Hopfian--co-Hopfian $p$-group.
What about mixed groups? In this section, we answer  the question. We start by recalling that	a group 	$G$ is called
	\begin{itemize}
		\item[(i)]
\emph{Hopfian} if its surjective endomorphisms are automorphisms;
		\item[(ii)] 
		\emph{co-Hopfian} if its injective endomorphisms are automorphisms.
	\end{itemize}

In what follows we will use the following two items:
\begin{fact}\label{dir}	\begin{itemize}
		\item[(i)] (See \cite[Claim 2.15(1)]{Sh:F2005}). Any direct summand of a  co-Hopfian  abelian group
is again co-Hopfian.

	\item[(ii)]
(See  \cite[Theorem 1.2]{1214}).
Suppose  $ 2^{\aleph_{0}}<\lambda< \lambda^{\aleph_{0}}$. Then  there is
no co-Hopfian abelian group of size $\lambda$.
\end{itemize}
	\end{fact}
Here, we introduce a useful criteria:

\begin{definition}\label{a16}
	Let  $G$ be an abelian  group of size $\lambda$  and $m, n \geq 1$ be such that $m \mid n.$ Then:
	\begin{enumerate}
		\item $\rm{NQr}_{(m, n)}(G)$ means that there is an $(m, n)$-anti-witness $h$, which means:
		
		\begin{enumerate}[(a)]
			\item $h \in \End(\Gamma_{n}(G))$,
			
			\item $\range(h)$ is a bounded group,
			
			\item $h' := m\cdot\id_{\Gamma_{n}(G)} + h \in \End(\Gamma_{n}(G))$ is $1$-to-$1$,
			
			\item $h'$ is not onto or $m > 1$ and $G/\Gamma_{n}(G)$ is not $m$-divisible.
		\end{enumerate}	
		
		\item $\rm{NQr}_m(G)$ means $\rm{NQr}_{(m, n)}(G)$ for some $n \geq 1$.
		
		\item $\rm{NQr}(G)$ means $\rm{NQr}_m(G)$ for some $m \geq 1$.

	\end{enumerate}
\end{definition}

\begin{definition}\label{a16b}
	Adopt the previous notation.\begin{enumerate}
		\item $\rm{Qr}(G)$ means the negation of $\rm{NQr}(G)$.
		
		\item $\rm{Qr}_*(G)$ means $\rm{Qr}(G)$ and in addition that
		 $\Gamma_p(G)$ is unbounded, for at least one $p \in \mathbb{P}$.

	\end{enumerate}
\end{definition}

In items \ref{a19}--\ref{a31}  we check $\rm{NQr}_{(m, n)}(G)$ and its negation. This enables us to present   some  new classes of co-Hopfian and non co-Hopfian groups.

\begin{lemma}\label{a19}
	Let $G$ be an abelian group such that the property $\rm{NQr}(G)$ holds. Then $G$ is not co-Hopfian. Furthermore, let  $h\in\Hom(G, \Gamma_{n}(G))$ be such that $h \restriction \Gamma_{n}(G)$ is an $(m, n)$-anti-witness. Then $m \cdot\id_G + h$ witnesses that $G$ is not co-Hopfian.
\end{lemma}

\begin{proof}
	Suppose that $G$ admits an $(m, n)$-anti-witness $h_0\in\End (\Gamma_{n}(G))$ as in  Definition \ref{a16}. As $h_0$ is bounded, by Fact  \ref{a7} we can extend $h_0$ to $h_1 \in \rm{Hom}(G, \Gamma_{n}(G))$. So, the following diagram commutes:
	
	$$\xymatrix{
		&0 \ar[r]&\Gamma_{n}(G)\ar[r]^{\subseteq_\ast}\ar[d]_{h_0}&G\ar[dl]^{\exists h_1}\\
		&& \Gamma_{n}(G)
		&&&}$$

	 We claim that $f = m\cdot\rm{id}_G + h_1 \in \End(G)$ is $1$-to-$1$ but not onto.
	\begin{enumerate}
		\item[$(*)_1$]	$f$ is one-to-one.
	\end{enumerate}
	To see this, suppose $x \in G$ in non-zero and we want to show that $f(x) \neq 0.$ Suppose first we deal with  the case $x \in \Gamma_{n}(G) \setminus \{0\}$. According to clause (c) of Definition \ref{a16}(1) we have $$ f(x) = mx + h_1(x) = m \cdot  \id_{\Gamma_{n}(G)}(x) + h_0(x) \Rightarrow f(x) \neq 0.$$
	Now, suppose that $x \in G \setminus \Gamma_{n}(G)$.  Recall from Definition \ref{a16} that $m$ divides $n$. As $m \mid n,$ we have
	$mx \in G \setminus \Gamma_{n}(G)$. If $f(x)=0$, we have $mx + h_1(x)=0$, thus
	\[
	h_1(x)=-mx \in G \setminus \Gamma_{n}(G).
	\]
	But,
	$\rng(h_1) \subseteq \Gamma_{n}(G)$, which is impossible.    Thus $f$ is $1$-to-$1$, as wanted.

	\begin{itemize}
		\item[$(*)_2$]  $f$ is not onto.
	\end{itemize}
	For this, we  consider two cases:
	
	{Case 1}) $h_{0}$ is not onto:
	\newline  By the case assumption, there is $$y \in \Gamma_{n}(G) \setminus \range\big(\rm{id}_{\Gamma_{n}(G)} + (h_0 \restriction \Gamma_{n}(G))\big)$$ and it is easy to see that such a $y$ is  also a witness for $f$ to be not onto.

	{Case 2}) $h_0$ is onto:
	\newline By Definition \ref{a16}(1)(d), we must have
	$m > 1$ and $G/ \Gamma_{n}(G)$ is not $m$-divisible.
	Let $z \in G$ be such that $z + \Gamma_n(G)$ is not divisible by $m$ in $G/ \Gamma_m(G)$. Clearly, $z$ does not belong to $\range(f).$
	
	The lemma follows.
\end{proof}



\begin{lemma}\label{a22}
	Let  $K$ be an abelian  $p$-group.  The following claims are valid:
	 If $\NQr(K)$ holds, then $K$ is infinite.
\end{lemma}

\begin{proof}
	By definition, there are  $m$ and $n$    such that $m \mid n$ and that
	$\NQr_{(m, n)}(K)$ holds. Thanks to Definition \ref{a16}(1), there is $h \in \End(\Gamma_{n}(G))$
satisfying the following properties:
	\begin{enumerate}
\item[(a)] $\range(h)$ is a bounded group,
		\item[(b)] $h' := m\cdot(\rm{id}_{\Gamma_{n}(K)}) + h \in \End(\Gamma_{n}(K))$ is $1$-to-$1$,
		\item[(c)] $h'$ is not onto or $m > 1$ and $K/\Gamma_{n}(K)$ is not $m$-divisible.
	\end{enumerate}	
	We have two possibilities:
	1) $p \nmid n$ and 2) $p \mid n$.
	
\begin{enumerate}
	\item[(1)]Suppose first that
	$p \nmid n$. 	As $K$ is a $p$-group,  $\Gamma_{n}(K) = \{ 0 \}$. This means that $h$ is constantly zero and is onto, as well as $h'$. Thanks to clause (c) it follows   that $m > 1$ and $K$ is not $m$-divisible. Since $m \mid n$ we deduce that $p \nmid m$.
Now, we consider the map $m\cdot\id_K:K\to K$.  Since $K$ is not $m$-divisible, this map is not surjective. Let us show that it is 1-to-1. To this end, let $x\in K$ be such that $mx=0$. Let $\ell$ be the order of $x$ so that $p^\ell x=0$. As $(p^\ell, m)=1,$ we can find $r,s$ such that $r p^\ell+ sm=1.$ By multiplying both sides with $x$, we obtain
$$x=r p^\ell x+ smx=0+0=0.$$
 It follows that $m \cdot \id_K: K \to K$ is 1-to-1 and not onto, hence $K$ is infinite.
		\item[(2)]Suppose $p \mid n$.  As $K$ is a $p$-group, this implies that  $\Gamma_{n}(K) = K$. Therefore, in the above item (c), the case ``$K / \Gamma_{n}(K)$ is not $m$-divisible{''}  does not occur. This is in turn implies that $h'$ is not onto $K$.  We proved that the map
		$h'\in\End(K)$ is 1-to-1 and not onto. Hence $K$ is infinite.
	\end{enumerate}	
	The proof is now complete.
\end{proof}

\begin{discussion}
	Keep the notation of Fact \ref{kapla2}. One can not   replace "divisible" with "reduced" and drives a similar result, as some easy examples suggest this. Here, we consider this as an application of the construct of co-Hopfian groups. 
	
(i)	Suppose on the way of contradiction that the replacement is valid. 

(ii) Let $G$ be  a co-Hopfian group such that its reduced part is unbounded (recall from the introduction that a such group exists, see \cite{crawley}).

(iii) Here, we drive a contradiction by showing from that $G$ is not co-Hopfian.
Indeed, let $K_2$ be the maximal divisible subgroup of $K$. Recall from Fact \ref{ided} that $K_2$ is injective. Since it is injective, we know $K_2$ is a directed summand. Let us write $K$ as $K=K_1 \oplus K_2.$ Due to the maximality of $K_2$ one may know that $K_1$ is reduced.
We show that $K_1$ is not co-Hopfian, and hence by Fact \ref{dir}(i), $K$ is not co-Hopfian. Thus by replacing $K$ by $K_1$ if necessary, we may assume without loss of generality that $K$ is reduced and unbounded. For $\ell < \omega,$ we choose by  induction $H_{\ell}, y_\ell$ and $z_\ell$ such that:

\begin{enumerate}
	\item[($iii)_a$] $H_{0} = K, $
	
	
	\item[($iii)_b$] if $\ell = k+1,$ then $H_{k} = H_{\ell} \oplus \mathbb{Z}z_{\ell},$
	
	\item[($iii)_c$] $z_\ell \in (\mathbb{Z}y_\ell)_*$ recall that $(\mathbb{Z}y_\ell)_*$ denotes the pure closure of
	$\mathbb{Z}y_\ell$,
	
	\item[($iii)_d$] $y_{\ell+1} \in H_\ell$,
	\item[($iii)_e$] The order of $z_i$ is $\geq p^{\ell}$.
	
	
	
\end{enumerate}
[Why? For $\ell=0$, we  set 	$H_{0} = K$ and let $y_0 \in K$ be arbitrary.
Then $(\mathbb{Z}y_0)_*$ is a pure subgroup of $K$ of bounded exponent. Thanks to   Fact \ref{kapla2} we know
$(\mathbb{Z}y_0)_*$ is a direct summand of $K$. In view of Fact \ref{bounded_exponent}  we can find $z_0$ such that $\mathbb{Z}z_0$ is a direct summand of
$(\mathbb{Z}y_0)_*$. In other words, $\mathbb{Z}z_0$ is a direct summand of $H_0=K$ as well. Consequently,
we have $H_0=H_1 \oplus \mathbb{Z}z_0$  for some $H_1$.
Having defined  inductively $\{H_\ell, y_\ell,z_\ell\}$, let $y_{\ell+1} \in H_\ell$.
Let  $\chi$ be a regular cardinal, large enough, so that $H_\ell \in  \cH(\chi) $. The notation $\cB$ stands for $(\cH(\chi), \in)$. Let $\cB_\ell$ be countable
such that $H_\ell\in\cB_\ell$.
Now, we look at  $$\mathcal{L}_\ell:=\cB_\ell\cap H_\ell.$$ We find easily that $\mathcal{L}_\ell$ is an unbounded countable abelian $p$-group. Hence it is of the form $\oplus_{i}\mathbb{Z}z_{\ell,i}$ where $z_{\ell,i}$ is of order $p^{m(\ell,i)}$.
As $\mathcal{L}_\ell$ is unbounded,  we may and do assume that $m(\ell,i)>\ell$. This implies that
$\mathbb{Z}z_{\ell,i}$ is a pure subgroup of
$\mathcal{L}_\ell$, and hence $H_\ell$. Consequently,
$\mathbb{Z}z_{\ell,i}$ is a direct summand  of
$H_\ell$  as well.
By definition,
we have $H_\ell=H_{\ell+1} \oplus \mathbb{Z}z_{\ell+1}$ for some abelian  subgroup $H_{\ell+1}$ of $H_\ell$.]

For each $i<\omega$, we let $\ell(i)>1$ be such that  $z_{i}$ is of order $p^{\ell(i)}.$
Following clause $(e)$,  clearly we can find some infinite $u \subseteq \omega$
such that the sequence $\langle \ell(i): i \in u     \rangle$ is increasing.
For any $j<\omega,$ we clearly have
$ \bigoplus_{i \in u \cap j} \bbZ z_{i} \subseteq_{*}  K,$ 
and hence
 $\oplus_{i \in u} \bbZ z_{i} \subseteq_{*}  K.$ 
In the light of part (i), $\bigoplus_{i \in u} \bbZ z_{i}$ is a direct summand of $K$, thus there is  some $K_{3}$ such that  $K = \oplus_{i \in u} \bbZ z_{i} \oplus K_{3}.$  Let   $\langle j(k): k < \omega \rangle$ be lists $u$ in an increasing order,
and define  $h \in \End(K)$ be such that 	
\begin{itemize}
	\item
	$h \rest K_{3} = \id_{K_{3}}$,
	
	\item $ h(z_{j(k)}) = p^{\ell(k+1)-1} z_{j(\ell +1)}$.
\end{itemize} It is easy to check that $h$ is a well-defined endomorphism of $K$ and it satisfies the following properties:
\begin{itemize}
	\item
	$h$ is injective,
	
	\item $h$ is not surjective.
\end{itemize}
In sum, $h$ witnesses that $K$ is not co-Hopfian.
This is a contradiction that we searched for it.
\end{discussion}
The following is clear:
\begin{corollary}
Let $G$ be  a   $p$-group such that its reduced part is unbounded and its countable pure subgroups
are directed summand. Then $G$ is not co-Hopfian.
\end{corollary}
\begin{lemma}\label{a25}
	Let  $G$ be an abelian  group  of size $\lambda$  and $m \geq 1$. Suppose there is a bounded  $h \in \End(G)$ such that $f := m\cdot\id_{G} + h \in \End(G)$ is $1$-to-$1$ not onto \footnote{Thus $f$ witnesses non co-Hopfianity of $G$.}. Then for some $n \geq 1$ we have:
	
	\begin{enumerate}
		\item[(i)] $\NQr_{(m, n)}(G)$,
		
		\item[(ii)] Letting $h_0 = h \restriction \Gamma_{n}(G)$, $h_0$ is an $(m, n)$-anti-witness for $\Gamma_{n}(G)$.
	\end{enumerate}
\end{lemma}

\begin{proof}
	
	Let $f$ and $h$ be as above. As $\range(h)$ is bounded, for some $n \geq 1$ we have $\range(h) \leq \Gamma_{n}(G)$ and without loss of generality $m \, \vert \, n$. Possibly, replacing $n$ with $nm$, which is possible as $n_{1} \vert n_{2}$ implies that $\Gamma_{n_{1}}(G) \leq\Gamma_{n_{2}}(G)$. Notice now that:
	
	\begin{enumerate}
		\item[$(*)_{1}$]
		\begin{enumerate}
			\item[(a)] $f$ maps $\Gamma_{n}(G)$ into itself.
			
			\item[(b)] if $x \in G \setminus \Gamma_{n}(G)$, then $f(x) \notin \Gamma_{n}(G)$.
		\end{enumerate}	
	\end{enumerate}	
	Clause	(a) clearly holds as by the choice of $n$ we have $\range(h) \leq \Gamma_{n}(G)$. To see clause (b), we suppose by contradiction
	that $f(x)=mx+h(x) \in \Gamma_{n}(G).$ It follows that $mx=f(x)-h(x) \in \Gamma_{n}(G)$, and hence as $m \, \vert \, n$,
	$x \in \Gamma_{n}(G),$ a contradiction.
	
	Let now $h_{0} = h \restriction \Gamma_{n}(G).$ Then  we have:
	
	\begin{enumerate}
		\item[$(*)_{2}$]
		\begin{enumerate}
			\item $h_{0} \in \End(\Gamma_{n}(G)),$
			
			\item $h_{0}$ is bounded,
			
			\item Since $f$ is $1$-to-$1$, so is $f_0 = m\cdot\id_{\Gamma_{n}(G)} + h_{0} \in \End(\Gamma_{n}(G))$.
		\end{enumerate}
	\end{enumerate}
	We are left to show that  $h_0$ is an $(m, n)$-anti-witness. By $(*)_{2}$ it suffices show that $f_0$ is not onto or $G / \Gamma_{n}(G)$ is not $m$-divisible. Suppose on the contrary that $f_0$ is onto and $G / \Gamma_{n}(G)$ is $m$-divisible. We are going to show that $f$
	is onto, which contradicts our assumption. To this end, let $x \in G.$ Since $G / \Gamma_{n}(G)$ is $m$-divisible,
	we can find some $y \in G$ such that
	\[
	x-my \in \Gamma_n(G).
	\]
We look at $$w:=x-my-h_0(y)\in\Gamma_{n}(G).$$	As $f_0$ is onto, we can find some $z \in \Gamma_n(G)$ such that $f_0(z)=w$. So,
	\[
	x-my-h_0(y)=w=f_0(z)=mz+h_0(z).
	\]
	Using this equation, and the additivity of $h_0$, we observe that $$x=m(y+z)+h_0(y+z)=f(y+z).$$
In other words, $f$ is onto. This is a contradiction.	
\end{proof}
\begin{notation}	Suppose $\kappa$ and $\mu$
		are infinite cardinals. The infinitary language $\mathcal{L}_{\mu, \kappa}(\tau)$
		is defined so as its vocabulary is the same as $\tau,$ it has the same terms and atomic formulas as in $\tau,$ but we also allow conjunction and disjunction of length less than $\mu$, i.e., if $\phi_j,$ for $j<\beta < \mu$ are formulas, then so are $\bigvee_{j<\beta}\phi_j$ and $\bigwedge_{j<\beta}\phi_j$. Also, quantification over less than $\kappa$ many variables. 
\end{notation}

\begin{lemma}\label{a28}
	
	Let $G$ be a reduced abelian group of size $\lambda$ such that
	\begin{enumerate}
		\item  $\lambda > 2^{\aleph_0}$,  \item $G$ is   co-Hopfian.
	\end{enumerate}
		Then the property $\Qr_*(G)$ is valid.
\end{lemma}

\begin{proof}
  Thanks to Lemma  \ref{a19} we know  $\Qr(G)$ is satisfied, so it is enough to show that for some prime $p$,  $\Gamma_p(G)$ is not bounded. Towards a contradiction, we suppose that  $\Gamma_p(G)$ is bounded   for every prime $p \in \bbP$.
	
Here, we are going to show the pure subgroup $\Gamma_p(G)$ is finite.	
Suppose on the way of contradiction that  $\Gamma_p(G)$ is infinite. Recall that $p$-torsion subgroups are pure. According to
Fact \ref{kapla1} $\Gamma_p(G)$  is a direct summand of $G$, as we assumed that it is bounded. Also, following Fact \ref{bounded_exponent}
 we know that $\Gamma_p(G)$  is a direct summand of cyclic groups. In sum, we observed that $\Gamma_p(G)$ has a direct summand $K$ which is a countably infinite $p$-group.  In view of Fact \ref{co-hop_pgroup}(i), we may and do assume that $K$ is not co-Hopfian. Recall that any direct summand of co-Hopfian, is co-Hopfian. This means that $G$ is not co-Hopfian as well, which contradicts our assumption. Thus, it follows that for every $p \in \mathbb{P}$, the group $\Gamma_p(G)$ is finite and therefore a direct summand of $G$, hence there is a projection $h_p$ from $G$ onto $\Gamma_p(G)$.
Recall that $p \in \mathbb{P}$ and also $h_p \restriction \Gamma_p(G) \in \End(\Gamma_p(G))$ is essentially equal to the identity map, so is one-to-one, and hence onto, as $\Gamma_p(G)$ is finite.    Since $\Qr(G)$ is satisfied, it follows from Definition  \ref{a16}(1)(d) that $G/ \Gamma_p(G)$ is $p$-divisible.
	
	Now, we take $\chi$ be a  regular cardinal, large enough,  such that $G \in \cH(\chi)$ and let $$M \prec_{\mathcal{L}_{\aleph_1, \aleph_1}} (\cH(\chi), \in)$$ be
	such that:
	\begin{itemize}
		\item $M$ has cardinality $2^{\aleph_0}$,
		\item  $G, \tor(G) \in M$,
		\item $2^{\aleph_0} + 1 \subseteq M$.
	\end{itemize}
In the light of  Fact \ref{co-hop_pgroup}(ii), we may and do assume that  $|\tor(G)| = \mu \leq 2^{\aleph_0}$. Recall that  $2^{\aleph_0} + 1 \subseteq M$ and $\tor(G) \in M$. These imply that   $\tor(G) \subseteq M$. Now, as $G/ \Gamma_p(G)$ is $p$-divisible, then so is $$\frac{G/ \Gamma_p(G)}{(G \cap M)/\Gamma_p(G)},$$ which by the Third Isomorphism Theorem, is canonically isomorphic to $G/G \cap M$.  As  $\tor(G) \subseteq M$,
	$G/(G \cap M) $ is torsion-free,  it is divisible. Let $x \in G \setminus M$ and define the sequence $(x_n : n < \omega)$  such that:
	\begin{itemize}
		\item	$x_0 = x$,
		\item If $n = m+1$ then $$G/(G \cap M) \models ``n!x_n+(G \cap M) = x_m+(G \cap M)''.$$
	\end{itemize}
	So, letting $a_0 = 0$ and for $n=m+1 < \omega$, $$a_n = n!x_n - x_m \in G \cap M,$$ we have that $(a_n : n < \omega) \in M^\omega \subseteq M$ and so, as $$M \prec_{\mathcal{L}_{\aleph_1, \aleph_1}} (\cH(\chi), \in),$$ we can find $$\bar{y} = (y_n : n < \omega) \in (G \cap M)^\omega$$ such that $a_n = n!y_n - y_m$, but then for every $m < \omega$:
	$$G \models``m!(x_{m+1} - y_{m+1}) = x_m - y_m''.$$
	Hence, $$\bigcup \{\mathbb{Z}(x_{m} - y_{m}) : m < \omega\}$$ is a non-trivial divisible subgroup of $G$, contradicting the assumption that $G$ is reduced. So we have proved the desired claim.
\end{proof}

\begin{proposition}\label{a31}
	Let $G \in $ be a boundedly endo-rigid abelian group. The following assertions are valid:
	\begin{enumerate}
		\item $G$ is co-Hopfian iff $\Qr(G),$
		
		\item If $|G| > 2^{\aleph_0}$, then $G$ is co-Hopfian iff $\Qr_*(G)$.
	\end{enumerate}
\end{proposition}

\begin{proof}
	(1). If $G$ is co-Hopfian, then by Lemma \ref{a19}, $\Qr(G)$ holds. For the other direction, suppose that $G$ is boundedly rigid and $\Qr(G)$ holds. Let $f \in \End(G)$ be $1$-to-$1$, we want to show that $f$ is onto. As $G$ is boundedly rigid we have $m$, $h$ and $L$ such that the following items hold:
	
	\begin{enumerate}
		\item[]
		\begin{enumerate}
			\item[(a)] $m \in \mathbb{Z}$, $h \in \rm{End}(G),$
			
			\item[(b)] $f(x) = mx + h(x),$
			
			\item[(c)] $L = \range(h)$ is a bounded subgroup of $G$ (and so of $\tor(G)$).
		\end{enumerate}
	\end{enumerate}
	If $f$ is not onto, then by Lemma \ref{a25}, there is $n \geq 1$ such that $\NQr_{(m, n)}(G)$ holds, which is not possible (as we are assuming $\Qr(G)$). Thus $f$ is onto as required.

	(2). It follows from clause (1) and Lemma \ref{a28}.
\end{proof}

\begin{construction}\label{a35}
Let $K := \oplus \{\frac{\mathbb{Z}}{p^n\mathbb{Z}} : p \in \mathbb{P} \text{ and } 1 \leq n <m \}$, where  $m< \omega$,
and  $\mathbb{P}$ is the set of prime numbers. Let $G$ be a boundedly endo-rigid abelian group   such that $\tor(G)=K$\footnote{In the light of our main result such a group exists for any $\lambda=\lambda^{\aleph_0} > 2^{\aleph_0}$ and the size of $G$ should be $\lambda$.}. Then $G$ is co-Hopfian.	
\end{construction}

\begin{proof}
For any $p_1\in \mathbb{P}$ and $n_1<m$, let us define 	$$
(x_{(p_1, n_1)})_{(p,n)}=\left\{\begin{array}{ll}
1+ p^n\mathbb{Z}&\mbox{if } (p,n)=(p_1,n_1)\\
0	&\mbox{otherwise }
\end{array}\right.
$$ For simplicity, we abbreviate it by $x_{(p_1, n_1)}$.	Assume towards a contradiction that there exists $f \in \End(G)$ such that $f$ is $1$-to-$1$ and not onto. As $G$ is boundedly endo-rigid, there are $m \in \mathbb{Z}$ and $h \in \BEnd(G)$ such that $f = m\cdot\id_G + h$.
	As $f$ is $1$-to-$1$ and $K$ has no infinite bounded subgroup, we can conclude that $m \neq 0$.
	\begin{enumerate}
		\item[$(*)_{1}$] $m \in \{1, -1\}$.
	\end{enumerate}
	To see $(*)_{1}$, suppose on the contrary that there is $p \in \mathbb{P}$ such that $p | m$ and let $m_1$ be such that $m = m_1p$. Now, as $\range(h)$ is bounded, there is $k \geq 1$ such that $$p^k(\range(h)) \cap \Gamma_{p}(G) = \{0\}.$$ Let $n \geq k+1$, then:
	
	\begin{equation*}
	\begin{array}{clcr}
	f(p^{n-1}x_{(p, n)}) &= m p^{n-1}x_{(p, n)} + h(p^{n-1}x_{(p, n)})\\
	&
	= m_1pp^{n-1}x_{(p, n)} + p^kh(p^{n-1 - k}x_{(p, n)})\\
	&=0,
	\end{array}
	\end{equation*}
	which contradicts the fact that $f$ is $1$-to-$1$. This completes the argument of $m \in \{1, -1\}$
	and without loss of generality we may assume that $m = 1$. Thus $f=\id_G+h$.
	
	\begin{enumerate}
		\item[$(*)_{2}$] $f$ maps $G \setminus \tor(G)$ into itself.
	\end{enumerate}
	This is because $f$ is $1$-to-$1$. Indeed let $x \in G \setminus \tor(G)$. If $f(x) \in \tor(G),$
	then for some $k, f(kx)=kf(x)=0,$ thus $kx=0,$ i.e., $x \in \tor(G)$ which contradicts
	$x \in G \setminus \tor(G)$.
	\begin{enumerate}
		\item[$(*)_{3}$] $f \restriction \tor(G) \in \End(\tor(G))$ is $1$-to-$1$ not onto.
	\end{enumerate}
	Clearly $f \restriction \tor(G) \in \End(\tor(G))$, and since $f$ is $1$-to-$1$,   $f \restriction \tor(G)$ is $1$-to-$1$ as well. Now, suppose by contradiction that $f \restriction \tor(G)$ is onto. Then
	
	\begin{enumerate}
		\item $\tor(G) \subseteq \range(f),$
		
		\item $x \in G \Rightarrow f(x) = x + h(x) \in  \tor(G).$ 	
	\end{enumerate}Recall that $h(x) \in  \tor(G).$ Apply this along with $(1)$, we deduce that $h(x)\in \range(f).$ Also, recall that $\range(f)$ is a group. Now, let
$x \in G$. Thanks to $(2)$,  we observe that$$ x =f(x)- h(x) \in  \range(f).$$ In other words, $f$ is onto, a contradiction. So, $f \rest \tor(G)$ is not onto.
	
	\begin{enumerate}
		\item[$(*)_{4}$]
		\begin{enumerate}
			\item[(a)] for every $p \in \mathbb{P}$, $f$ maps $\Gamma_p(G)$ into itself and so $f \restriction \Gamma_p(G)$ is $1$-to-$1,$
			
			\item[(b)] for some $p \in \mathbb{P}$, $f \restriction \Gamma_p(G)$ is not onto.
		\end{enumerate}
	\end{enumerate}
	Item (a) above is simply because $f$ is $1$-to-$1$. To see (b) holds, note that if  $f \restriction \Gamma_p(G)$
	is onto for all prime number $p$, then so is $f \restriction \tor(G)$, which contradicts $(*)_3.$
	
	Thus, let us fix some prime $p \in \mathbb{P}$ such that $f \restriction \Gamma_p(G)$ is not onto
	and let $h_p = h \restriction \Gamma_p(G)$. Then by the above observations,  it equipped with the following properties:
	\begin{enumerate}
		\item[$(*)_{5}$]
		\begin{enumerate}
			\item[(a)] $h_p \in \rm{End}(\Gamma_p(G)),$
			
			\item[(b)] $\range(h_p)$ is bounded,
			\item[(c)] $h'_p = m\cdot \id_{\Gamma_p(G)} + h_p = \id_{\Gamma_p(G)} + h_p$ is $1$-to-$1$,
			
			\item[(d)] $h'_p$ is not onto.
		\end{enumerate}
	\end{enumerate}
In the light of Definition~\ref{a16} and $(*)_{5}$  we observe that
	\begin{enumerate}
		\item[$(*)_{6}$] $h_p$ is a $(1, p)$-anti-witness for $\Gamma_p(G)$ and so $\NQr(\Gamma_p(G))$.
	\end{enumerate}
Thanks to Lemma \ref{a22}, $\Gamma_p(G)$  is infinite.
	But,
	\begin{center}
		$\Gamma_p(G) = \Gamma_p(K) =\bigoplus \{\frac{\mathbb{Z}}{p^n\mathbb{Z}} : 1 \leq n <m \}$,
	\end{center}
	which is finite. Thus we get a contradiction, and hence $f$ is onto. It follows that $G$
	is co-Hopfian and the lemma follows.
\end{proof}

\begin{corollary} \label{end}
	For any cardinals $\lambda> 2^{\aleph_{0}}$, there is a co-Hopfian abelian group $G$ of size $\lambda$ iff $\lambda= \lambda^{\aleph_{0}}$.
\end{corollary}

\begin{proof} Let $\lambda> 2^{\aleph_{0}}$
	be given. Suppose first that $\lambda< \lambda^{\aleph_{0}}$. In other words, $ 2^{\aleph_{0}}<\lambda< \lambda^{\aleph_{0}}$. According to  Fact \ref{dir}(ii),  there is
	no co-Hopfian  abelian  group of size $\lambda$. Now,  assume that $\lambda= \lambda^{\aleph_{0}}$. Let
	$$K := \oplus \{\frac{\mathbb{Z}}{p^n\mathbb{Z}} : p \in \mathbb{P} \text{ and } 1 \leq n <m \},$$ where  $m< \omega$. In the light of Theorem \ref{th2}, there exists a
	boundedly endo-rigid abelian group $G$ with $\tor(G) = K$. By Construction \ref{a35}, $G$ is
	co-Hopfian.\end{proof}

\begin{lemma}\label{a1d}
	Let $G=G_1\oplus G_2$ be a boundedly endo-rigid abelian  group. Then $G_1$ is  boundedly endo-rigid.
\end{lemma}
\begin{proof}
Let	$f_1 \in \End(G_1)$. Then $f_1\oplus\id_{G_2} \in \End(G)$. Since
 $G$ is boundedly endo-rigid there is $m \in \bbZ$ such that the map $x \mapsto f(x) - mx$ has bounded range. In other words, $$(f_1-m\cdot\id_{G_1})\oplus 0\subseteq(f_1-m\cdot\id_{G_1})\oplus(\id_{G_2}-m\cdot\id_{G_2})=(f-m\cdot \id_{G})$$has bounded range. By definition, 
$G_1$ is  boundedly endo-rigid.\end{proof}
\begin{notation} (Harrison) For each group $G$, we set $$S := S_{G}: = \{p \in \mathbb{P}:  G / \Gamma_{p}(G)\emph{ is not  p-divisible}\}.$$
\end{notation}

Now, we are ready to present the following promised  criteria:
\begin{proposition}\label{s11c} Let $\lambda > 2^{\aleph_{0}},$
and suppose $G$ is a boundedly endo-rigid abelian  group  of size $\lambda$.  Then $G$ is co-Hopfian if and only if:
	
	\begin{enumerate}
		\item[(a):] $S_{G}$ is a non-empty set of primes,
		
		\item[(b):]
		\begin{enumerate}
			\item[$(b_1)$] $\tor(G) \neq G,$
			
			\item[$(b_2)$] if $p \in S,$ then $ \Gamma_{p}(G)$ is not bounded,
			
			\item[$(b_3)$] if $\Gamma_p(G)$ is bounded, then it is finite (and $p \notin S_{G}$).
		\end{enumerate}
	\end{enumerate}
\end{proposition}

\begin{proof}
	Let $K := \tor(G)$, and for each prime number $p$, we set $K_{p} := \Gamma_{p}(G).$
	
	First, we assume that	
	$G$ is co-Hopfian, and we are going to show items (a) and (b) are valid. 
		As $G$ is co-Hopfian, and recall from the introduction that
	Beaumont and Pierce (see \cite{Be}) proved  that for the co-Hopfian group $G$, we know $\tor(G)$ is of size at most continuum.
 In other words,  $\vert \tor(G) \vert \leq 2^{\aleph_{0}}$. We combine this along with our assumption  $\vert G \vert = \lambda > 2^{\aleph_{0}},$ and conclude that $K = \tor(G) \neq G,$  as claimed by $(b_1)$.
	
	To prove $(b_2)$, let $p \in S$ and suppose by contradiction that   $K_{p}$ is bounded. As $K_p$ is pure in $G$,  and following Fact \ref{kapla1}, the boundedness property guarantees that $K_P$ is a direct summand of $G$. By definition, there is $G_p$  such that $G = K_{p} \oplus  G_{p}$. Now, we look at $\id_{K_{p}} + p \cdot \id_{G_{p}} \in \End(G)$. Let $$(k,g) \in \Ker(\id_{K_{p}} + p \cdot \id_{G_{p}}).$$ Following definition, $$(0,0)=(\id_{K_{p}} + p \cdot \id_{G_{p}})(k,g)=(k,pg).$$
	In other words,
	$k=0$ and as $G_p$ is $p$-torsion-free, $g=0$. This means that  $$\Ker(\id_{K_{p}} + p \cdot \id_{G_{p}})=0,$$ and hence $\id_{K_{p}} + p \cdot \id_{G_{p}}$ is $1$-to-$1$.
	Since $p\in S$, ${G_{p}}:=G / \Gamma_{p}(G)$ is not  p-divisible, thus there is  $g$ in $G_p$ such that
	$g\notin \range(p \cdot \id_{G_{p}})$. Consequently, $\id_{K_{p}} + p \cdot \id_{G_{p}}$
	is $1$-to-$1$ not onto. This is in contradiction with the co-Hopfian assumption, so $K_{p}$ is not bounded and  $(b_2)$ follows.
	
	In order to check $(b_3)$, suppose $K_{p} = \Gamma_{p}(G)$ is bounded. Then it is a direct summand of $G$, say $G = K_{p} \oplus G_{p}$. Since $G$ is co-Hopfian, and in view of Fact \ref{dir},  we observe that  $K_p$ is co-Hopfian. Thanks to Fact \ref{co-hop_pgroup}  $K_{p}$ is finite.
	
	Lastly, we check clause (a). Suppose on the way of contradiction that $S$ is empty. Let $G_{1} \prec_{\mathcal{L}_{\aleph_{1}, \aleph_{1}}} G$ be of cardinality $2^{\aleph_{0}}$ containing $\tor(G),$ recalling $\vert \tor(G) \vert \leq 2^{\aleph_{0}},$ so $G / G_{1}$ is divisible of cardinality $\lambda$.

	As $G_1 \neq G$, there is $x_{0} \in G \setminus G_{1},$ and note that $x \notin \tor(G)$. Now as $G/ \tor(G)$ is divisible, we can choose the sequence  $\langle x_{n}: n \geq 1 \rangle$ of elements of $G$, by induction on $n$, such that $x_0=x$ and for each $n$,
	$$G/\tor(G) \models`` n! x_{n+1}+\tor(G)= x_{n}+ \tor(G)''.$$
	Set $$ a_{n}:=n! x_{n+1}-x_n \in \tor(G).$$
	Note that $\langle  a_n: n<\omega  \rangle \in G_1$, thus as $G_{1} \prec_{\mathcal{L}_{\aleph_{1}, \aleph_{1}}} G$,
	we can find elements $y_{n} \in G_{1}$ for $n < \omega$  such that $$n! y_{n+1} = y_{n} + a_{n}.$$ Subtracting the last two displayed formulas,  shows that the group $$L = \bigcup \{\bbZ {(x_{n} - y_{n})}: n < \omega \}$$ is a  {non-zero} divisible subgroup  of  $G.$
	Recall from Fact \ref{ided} that $L$ is injective. Since it is injective, we know $L$ is a directed summand of its extensions.
In sum, the sequence $$0\longrightarrow L\stackrel{g} \longrightarrow G\longrightarrow \coker(g)\longrightarrow 0,$$ splits. Recall from
	Discussion \ref{endo} that $$\End(I)=\prod_{p\in\mathbb{P}_0} \widehat{\mathbb{Z}}_{p}^{\oplus{x_p}},$$	 where $\mathbb{P}_0:=\mathbb{P}\cup\{0\}$  and ${x_p}$  are some index sets. This turns out that $I$ is not boundedly endo-rigid, provided it is nonzero.
Recall from Lemma \ref{a1d} that the property of boundedly endo-rigid behaves well with respect to direct summand, it obviously implies $G$ is not boundedly endo-rigid. This contradiction implies that $S$ is not empty.
	So clause (a) holds. All together, we are done proving the left-right implication.
	
	For the right-left implication, assume items $(a)$ and $(b)$  hold, and we  show  that $G$ is co-Hopfian. Suppose on the way of contradiction  that there exists $f \in \End(G)$ such that $f$ is $1$-to-$1$ and not onto. As $G$ is boundedly endo-rigid, there are $m \in \bbZ$ and $h \in \BEnd(G)$ such that $f = m \cdot \id_{G} + h.$
	
	\sn
	\begin{enumerate}
		\item[$(*)_{1}$] $m \neq 0.$
	\end{enumerate}
	To see $(*)_{1}$, suppose  $m = 0.$ Then $f = h$, and since $\range(h)$ is bounded and $f$ is $1$-to-$1$, we can conclude that $G$ is bounded
	and  therefor $G = \tor(G)$.  This contradicts clause $(b_1)$.

	\begin{enumerate}
		\item[$(*)_{2}$]  If $\Gamma_{p}(G)$ is infinite, then $p \nmid m.$
	\end{enumerate}
	In order to see $(*)_2$, first note  that $\tor(G)$ is unbounded, as otherwise $\Gamma_{p}(G)$ is also bounded, hence by $(b_3)$ it is finite, contradicting our assumption.  Suppose on the way of contradiction that $p \mid m$. Then there is $m_1$ such that $m = m_{1}p$.
	Now, as $\range(h)$ is bounded, there exists $k \geq 1$ such that $$p^{k}\big(\range(h) \rest \Gamma_{p}(G)\big) = \{ 0 \}.$$ Recall that $K_{p}$ is unbounded. This gives us an element  $x \in \Gamma_{p}(G)$ of order $p^{n}$ for some $n \geq k +1.$ But then

	\begin{equation*}
	\begin{array}{clcr}
	f(p^{n-1}x)  &= mp^{n-1}x + h(p^{n-1}x) \\
	&
	=m_{1}pp^{n-1}x + p^{k}h(p^{n-1-k}x)\\
	&=0,
	\end{array}
	\end{equation*}
	
	which contradicts the fact that $f$ is $1$-to-$1$.

	As before,  we have the following properties:
	
	\begin{enumerate}
		\item[$(*)_{3}$] $f$ maps $G \setminus \tor(G)$ into itself.
		\item[$(*)_{4}$] $f \rest \tor(G) \in \End(\tor(G))$ is $1$-to-$1$ not onto.
		\item[$(*)_{5}$]
		\begin{enumerate}
			\item[(a)] for every $p \in \mathbb{P}$, $f$ maps $\Gamma_p(G)$ into itself and so $f \restriction \Gamma_p(G)$ is $1$-to-$1,$
			
			\item[(b)] for some $p \in \mathbb{P}$, $f \restriction \Gamma_p(G)$ is not onto.
		\end{enumerate}
	\end{enumerate}
	Fix $p \in \mathbb{P}$ such that $f \restriction \Gamma_p(G)$ is not onto. Then  $h_p := h \restriction \Gamma_p(G)$ is equipped with the following properties:
	\begin{enumerate}
		\item[$(*)_{6}$]
		\begin{enumerate}
			\item[(a)] $h_p \in \rm{End}(\Gamma_p(G)),$
			
			\item[(b)] $\range(h_p)$ is bounded,
			\item[(c)] $h'_p = m\cdot \id_{\Gamma_p(G)} + h_p = \id_{\Gamma_p(G)} + h_p$ is $1$-to-$1$,
			
			\item[(d)] $h'_p$ is not onto.
		\end{enumerate}
	\end{enumerate}
	In the light of its definition, $h_p$ is a $(1, p)$-anti-witness and so $\NQr(\Gamma_p(G))$ holds. Thanks to Lemma \ref{a22}:
	\begin{enumerate}
		\item[$(*)_{7}$] $\Gamma_p(G)$  is infinite.
	\end{enumerate}
This is in contradiction with
$(*)_2$.
\end{proof}

In  \cite{AGS} we   studied  absolutely co-Hopfian  abelian groups. Recall an abelian group is absolutely co-Hopfian if it is co-Hopfian
in any further generic extension of the universe.
Also, see \cite{1205} for the existence of absolutely  Hopfian abelian groups of any given size.
 Similarly, one may define absolutely endo-rigid groups.
Despite  its simple statement, one of the most frustrating problems in the theory
infinite abelian groups is as follows:
\begin{problem}
	Are there  absolutely endo-rigid abelian groups of arbitrary large cardinality?
\end{problem}

\begin{acknowledgement}
	The authors  thank the referees for reading the paper thoroughly and providing valuable comments.
\end{acknowledgement}

\end{document}